\documentclass[11pt]{amsart}

\usepackage{amssymb}
\usepackage{color}

\usepackage[hmargin={20mm,20mm},vmargin={15mm,20mm}]{geometry}
\usepackage{geometry}
\usepackage{graphicx}
\usepackage{color}

\bibstyle{plain}
\newtheorem{theorem}{Theorem}[section]
\newtheorem{observation}[theorem]{Observation}
\newtheorem{example}[theorem]{Example}
\newtheorem{problem}[theorem]{Problem}
\newtheorem{lemma}[theorem]{Lemma}
\newtheorem{proposition}[theorem]{Proposition}
\newtheorem{corollary}[theorem]{Corollary}

\newtheorem{definition}[theorem]{Definition}
\newtheorem{remark}[theorem]{Remark}

\newcommand{\tL}{\tilde{\Lambda}}
\newcommand{\La}{{\Lambda}}
\newcommand{\F}{\mathcal F}

\newcommand{\G}{\mathcal G}

\newcommand{\I}{\mathcal I}

\newcommand{\N}{\mathbb N}

\newcommand{\R}{\mathbb R}

\newcommand{\mH}{\mathcal{H}}
\newcommand{\mK}{\mathcal{K}}
\newcommand{\mA}{\mathcal{A}}

\newcommand{\on}{\operatorname}
\newcommand{\ha}{\hat{\;}}
\newcommand{\oN}{\overline{\N}}
\newcommand{\iY}{Y^{(\infty)}}

\title[Zero-dimensional compact metrizable spaces as attractors of GIFSs]{Zero-dimensional compact metrizable spaces as attractors of generalized iterated function systems}

\author[Filip Strobin --- {\L}ukasz Ma\'slnka]{Filip Strobin --- {\L}ukasz Ma\'slanka}

\address
{\textsc{Filip Strobin}\\
Institute of Mathematics\\
Lodz University of Technology\\
\L\'od\'z, POLAND} 

\email{filip.strobin@p.lodz.pl}

\address{\textsc{{\L}ukasz Ma\'slanka}\\
Institute of Mathematics\\
Lodz University of Technology\\
\L\'od\'z, POLAND} 

\email{lukasz\_maslanka@interia.pl}



\begin{document}

\begin{abstract}%
	In the last years, the problem of considering $0$-dimensional compact metrizable spaces as~attractors of iteration function systems has been undertaken by several authors, for example by~T.~Banakh, E.~Daniello, M.~Nowak and F.~Strobin. In particular, it was proved that such a space $X$ is homeomorphic to the attractor of some IFS iff it is uncountable or it is countable but the {scattered} height of $X$ is~successor ordinal. Also, it was shown that in this case, a space $X$ can be embedded into the real line~$\R$ as {the attractor} of an IFS on $\R$, as well as can be embedded as {a nonattractor of any IFS}.
\newline
Miculescu and Mihail in 2008 introduced the concept of a \emph{generalized iterated function system} (GIFS in~short), a particular extension of the classical IFS. The idea is that, instead of families of selfmaps of a metric space~$X$, GIFSs consist of maps defined on a finite Cartesian {$m$-th power} $X^m$ with values in $X$ (in such a case we say that a GIFS is \emph{of order} $m$). It turned out that a great part of the classical {Hutchinson theory} has natural counterpart in this GIFSs' framework. On the other hand, there are known only few examples of~fractal sets which are generated by GIFSs, but which are not IFSs' {attractors}.
\newline
In the paper we study $0$-dimensional compact metrizable spaces from the perspective of GIFSs' theory. We prove that each such space $X$ (in particular, countable with limit {scattered} height) is homeomorphic to the~attractor of some GIFS on the real line. Moreover, we prove that $X$ can be embedded into the real line $\R$ as {the attractor of some} GIFS of order $m$ and (in the same time) {a nonattractor} of any GIFS of order $m-1$, as well as it can be embedded as {a nonattractor of any GIFS}. {Then} we show that a relatively simple modifications of $X$ deliver spaces whose each connected component is ``big'' and which are GIFS's { attractors} not homeomorphic with IFS's {attractors}. Finally, we use obtained results to show that a generic compact subset of a Hilbert space is not {the} attractor of any Banach  GIFS.
\end{abstract}

\maketitle

\section{Introduction}
The classical Hutchinson theorem (proved by Hutchinson \cite{Hut} and popularized by Barnsley \cite{B}) from early 80s' states that if $X$ is a complete metric space and $\F$ is a finite family of Banach contractions (that is, selfmaps with the Lipschitz constants $\on{Lip}(f)<1$), then there is a unique nonempty and compact $A\subset X$ such that
\begin{equation}\label{filip1}
A=\bigcup_{f\in\F}f(A).
\end{equation}
In this setting, a finite family $\F$ of continuous selfmaps of $X$ is called an \emph{iterated function system} (IFS in short), and a set $A_\F$ satisfying (\ref{filip1}) is called {an \emph{attractor} or a \emph{fractal generated by $\F$}} {(commonly, \emph{fractal sets} are assumed to have nontrivial structure which can be witnessed for example by its fractional dimension; we use more general approach - IFS fractals are all attractors of IFSs)}. It is well known that the thesis of the {Hutchinson} theorem holds under weaker contractive assumptions on $f_1,{\dots},f_n$ (like these due~to~Browder \cite{Br} or Matkowski \cite{Mat}{,} see for example \cite{H}), which, in case of compact space $X$, reduce to {the} Edelstein contractivity (see \cite{E}):
\begin{equation*}
\forall_{x,y\in X,\;x\neq y}\;d(f(x),f(y))<d(x,y).
\end{equation*} We will call an IFS consisting of Banach contractions as a \emph{Banach IFS} and consisting of~weaker types of~contractive maps - as a \emph{weak IFS}. Also, for a given IFS $\F$, we denote $\on{Lip}(\F):=\max\{\on{Lip}(f):f\in\F\}$.\\ 
By a \emph{topological IFS fractal} {or a \emph{topological fractal} for short} (see \cite{K}, \cite{Mi2}; in \cite{K} it is called a {topological} self similar set) we will mean a compact Hausdorff space $X$ such that for some IFS $\F$, $X=\bigcup_{f\in\F}f(X)$ and for every sequence ${\{f_k\}_{k\in\N}}\subset \F$, the set 
$$
\bigcap_{k\in\N}f_{1}\circ{\dots}\circ f_{k}(X)
$$
(called sometimes a \emph{fibre}), is singleton. As was proved in \cite{BKNNS} and \cite{MM3}, $X$ is a~topological IFS fractal iff $X$ is homeomorphic to the attractor of some weak IFS (in particular, it is metrizable). Finally, let us note that topological IFS fractals are attractors of so-called \emph{topologically contracting} IFSs, studied in mentioned papers.\\

In the last years, there has been an effort to detect those sets (especially, subsets of Euclidean spaces) and topological spaces which are IFSs' {fractals} or topological fractals. In particular, {in \cite{N} Magdalena Nowak
proved the following theorem (see Section 3.1 for the definition of the scattered height).}
\begin{theorem}\cite[Theorem 2, Corollary 1]{N}\label{nowak1} Let $X$ be a countable compact metrizable space.\\
(1) The following conditions are equivalent:
\begin{itemize}
\item[(i)] $X$ is topological fractal;
\item[(ii)] $X$ can be embedded into the real line as {the} attractor of a Banach IFS;
\item[(iii)] the {scattered} height of $X$ is successor ordinal.
\end{itemize}
(2) If $X$ is infinite, then it can be embedded into the real line as {a} nonattractor of any weak IFS.
\end{theorem}
Recently, Banakh, Nowak and Strobin in \cite {BNS2} (see also \cite{BNS} for a bit weaker result) proved that, additionally:
\begin{theorem}\label{bns1}
Each uncountable $0$-dimensional compact metrizable space can be embedded into the real line as {the} attractor of a~Banach IFS that consists of two maps.
\end{theorem}
A bit earlier, D'Aniello and Steele \cite{DS} proved {an} analogous result, but the IFSs they constructed consist of more than two maps.

An interesting generalization of the {Hutchinson theory} of fractals was introduced by Miculescu and Mihail in~2008~-- instead of selfmaps of a metric space, they considered mappings defined on a finite Cartesian {power} of a~given space with values in that space.

Let $m\in\N$ and $(X,d)$ be a metric space. If not stated otherwise, on the Cartesian {$m$-th power} $X^m$ we will consider the maximum metric $d_m$, i.e., 
$$d_m((x_1, {\dots}, x_m), (y_1, {\dots}, y_m)) := \max\{d(x_1, y_1), {\dots}, d(x_m, y_m)\}.$$
We say that a map $f: X^m \to X$ is \emph{ a generalized Banach contraction} if the Lipschitz constant $\on{Lip}(f)<1$, that is, there exists $\alpha < 1$ such that for any $x,y\in X^m$,
$$d(f(x),f(y)) \leq \alpha \cdot d_m(x,y).$$
Miculescu and Mihail in \cite{MM1} and \cite{M} proved that if $X$ is complete, $m\in\N$ and $\F$ is a finite family of  generalized Banach contractions $f:X^m\to X$, then there is a unique nonempty and compact $A\subset X$ such that
\begin{equation}\label{filip2}
A=\bigcup_{f\in\F}f(A\times{\dots}\times A).
\end{equation}
In this setting, a family $\F$ of continuous maps $f:X^m\to X$ is called a \emph{generalized iterated function system of order $m$} (GIFS in short), and a unique nonempty and compact set $A_\F$ satisfying (\ref{filip2}) is called {the \emph{attractor} or the \emph{fractal generated by $\F$}}.\\
After the papers of Miculescu and Mihail, other aspects of the theory of GIFSs and their fractals were considered, especially by them, Secelean and Strobin{,} see for example papers \cite{MM1}, \cite{MM2},  \cite{Se}, \cite{SS}, \cite{SS2}, \cite{S} and references therein.
In particular, it was proved that GIFSs consisting of weaker contractive type mappings also generates a unique fractal sets (see \cite{SS}; also cf. \cite{MM1}). Again, for compact $X$, these conditions reduce to the following:
\begin{equation*}
\forall_{x,y\in X^m,\;x\neq y}\;d(f(x),f(y))<d_m(x,y).
\end{equation*}
If a GIFS consists of generalized Banach contractions, then we call it a \emph{Banach GIFS}, and if it consists of weaker type of generalized contractions - we call it a~\emph{weak GIFS}. Similarly, for a GIFS $\F$, we set $\on{Lip}(\F):=\max\{\on{Lip}(f):f\in\F\}$. \\
One of the problems considering GIFSs is the following:
\begin{center}Is the class of GIFSs' {attractors} essentially wider than the class of IFSs' {attractors}?\end{center}
and, related to it,
\begin{center}Which sets/spaces are { attractors of} some GIFS?\end{center} 
Several interesting examples were given. In \cite{MM2} it was observed that the Hilbert cube $I~:=~\prod_{k=1}^\infty\left[0,\frac{1}{2^k}\right]$ is~generated by a Banach GIFS $\F=\{f,g\}$ of order $2$ defined by
$$
f((x_k),(y_k)):=\left(\tfrac{1}{2}x_1,{\tfrac{1}{2}}y_1,{\tfrac{1}{2}}y_2,{\tfrac{1}{2}}y_3,{\dots}\right)$$ and $$g((x_k),(y_k)):=\left({\tfrac{1}{2}}x_1+{\tfrac{1}{2}},{\tfrac{1}{2}}y_1,{\tfrac{1}{2}}y_2,{\tfrac{1}{2}}y_3,{\dots}\right).
$$
On the other hand, it cannot be generated by any Banach IFS, as it~has infinite dimension. However, to our best knowledge, it is not known whether $I$ is {the attractor of a weak IFS or a topological fractal}.\\
In \cite{S}, for each $m\geq 2$, there is constructed a Cantor subset $C(m)$ of the plane, which is~generated by~some Banach GIFS on the plane of order $m$, but is not generated by any weak GIFS of order $m-1$. Also, there is constructed a Cantor set $C$ which is not {the} attractor of any weak GIFS. On the other hand, each such $C(m)$ and $C$ are homeomorphic to the Cantor ternary set, hence they are homeomorphic to the attractor of a Banach IFS (in particular, they are topological fractals).\\

The aim of this paper is to study $0$-dimensional compact  metrizable spaces from the perspective of GIFSs' theory. It is organized as follows.\\
In the first part of the next section we show some simple, but useful observations concerning GIFSs and we also prove that certain quotients of topological fractals are also topological fractals.\\
In Section 3 we construct a wide class of metric spaces, \emph{$(\La,b)$-spaces}, which will be key in the proofs of main results. Also, we prove some basic properties of these spaces.\\
In Section 4 we prove that each compact $0$-dimensional metrizable space $X$ can be embedded into the real line as the attractor of a Banach GIFS of order $2$. In particular, if $X$ is countable and its {scattered} height is limit ordinal, then we obtain {the} attractor of a Banach GIFS which is not a topological fractal.\\
In the main result of Section 5 we prove that for any $m\geq 2$, each infinite compact metrizable $0$-dimensional space~$X$ can be embedded into the real line as the attractor of a Banach GIFS of order $m$ and {a nonattractor} of any~weak GIFS of order $m-1$, and also $X$ can be embedded into the real line as {a nonattractor} of any~weak~GIFS. This extends both mentioned results from \cite{S} and part (2) of Theorem \ref{nowak1}.\\
In Section 6 we show that, replacing points in certain countable compact spaces by appropriately ``big" sets, we obtain next examples of Banach GIFSs' {fractals}  which are not topological fractals.\\
Finally, in Section 7 we use earlier machinery to prove that a generic compact subset of a Hilbert space is not {the} attractor of any Banach GIFS.

\section{Basic definitions, observations and auxiliary constructions}
\subsection{Remarks on IFSs and GIFSs}
Here we make some simple observations which we will use later. The first one is obvious and we skip the proof.
\begin{lemma}\label{nl1}
Let $X$ be a metric space, $m_1,{\dots},m_n\in\N$ and for $i=1,{\dots},n$, $f_i:X^{m_i}\to X$ be a generalized Banach contraction. Take $m\geq \max\{m_1,{\dots},m_n\}$ and for each $i=1,{\dots},n$, define $\tilde{f}_i:X^m\to X$ by
$\tilde{f}_i(x_1,{\dots},x_m):=f_i(x_1,{\dots},x_{m_i})$,
and set $\tilde{\F}:=\{\tilde{f}_1,{\dots},\tilde{f}_n\}$. Then:\\
(1) $\tilde{\F}$ is a Banach GIFS of order $m$ and $\on{Lip}(\tilde{\F})=\max\{\on{Lip}(f_1),{\dots},\on{Lip}(f_n)\}$.\\
(2) If $A\subset X$ is nonempty and compact, and $A=\bigcup_{i=1}^nf_i(A^{m_i})$, then $A$ is the attractor of $\tilde{\F}$.
\end{lemma}
The second lemma shows that in a certain cases, we can extend maps from a given GIFS to some wider space.
\begin{lemma}\label{nl2}
Assume that $H$ is a Hilbert space, $X\subset H$ is nonempty and compact, and $\F$ is a Banach GIFS on $X$ of order $m$ with $\on{Lip}(\F)<\frac{1}{\sqrt{m}}$ so that $X$ is its attractor. Then for every $f\in\F$, there is its extension $\tilde{f}:H^m\to H$ such that $\tilde{\F}:=\{\tilde{f}:f\in\F\}$ is a Banach GIFS with $\on{Lip}(\F)\leq \sqrt{m}\on{Lip}(\F)<1$, whose attractor is $X$.
\end{lemma}
\begin{proof}Take $f\in\F$, and 
denote by $\rho_m$ the $\ell_2$ metric on $H^m$, that is, $\rho_m((x_1,{\dots},x_m),(y_1,{\dots},y_m)):=\sqrt{d(x_1,y_1)^2+{\dots}+d(x_m,y_m)^2}$. Clearly, \begin{equation}\label{vn1}d_m\leq \rho_m\leq \sqrt{m}d_m\end{equation}
Hence $\on{Lip}_{\rho_m}(f)\leq \on{Lip}(f)$, and since $(H^m,\rho_m)$ is a Hilbert space (more precisely, the metric $\rho_m$ is~generated by appropriate norm), using the Kirszbraun-Valentine theorem (see \cite[Theorem 1.12]{L}), we can extend the map $f$ to~the~map $\tilde{f}:H^m\to H$ so that $\on{Lip}_{\rho_m}(\tilde{f})=\on{Lip}_{\rho_m}(f)\leq \on{Lip}(f)$. Then, again by (\ref{vn1}), $\on{Lip}(\tilde{f})\leq \sqrt{m}\on{Lip}(f)$. Hence the~GIFS $\tilde{\F}:=\{\tilde{f}:f\in\F\}$ satisfies $\on{Lip}(\tilde{\F})\leq \sqrt{m}\on{Lip}(\F)$. The second part of the thesis is obvious.
\end{proof}
\begin{remark}\label{rem:ell}{
Let us remark that similar result holds for $\ell^\infty$ space { - if $K\subset \ell^\infty$, then any map $f:K^m\to\ell^\infty$ with $\on{Lip}(f)<\infty$ can be extended to a~map $\tilde{f}:(\ell^\infty)^m\to\ell^\infty$ with $\on{Lip}(\tilde{f})=\on{Lip}(f)$ (see \cite[Lemma 1.1(ii)]{L})}}.
\end{remark} 

The next lemma shows that if two GIFS's {fractals} are appropriately separated, then their union is also a GIFS {fractal}.
\begin{lemma}\label{filipnowy9}
Assume that $X$ is a compact metric space of the form $X=X_1\cup{\dots}\cup X_n$, such that
\begin{itemize}
\item[(a)] $X_1,{\dots},X_n$ are attractors of some Banach GIFSs of order $m$, $\F_1,{\dots},\F_n$, respectively;
\item[(b)] there are projections $\pi_i:X\to X_i$, $i=1,{\dots},n$, such that $\max\{\on{Lip}(\pi_i)\cdot\on{Lip}(\F_i):i=1,{\dots},n\}~<~1$.
\end{itemize}
Then $X$ is the attractor of some Banach GIFS $\mH$ of order $m$ with $\on{Lip}(\mH)\leq\max\{\on{Lip}(\pi_i)\cdot\on{Lip}(\F_i):i=1,{\dots},n\}$, and which consists of certain extensions of maps from $\F_1\cup{\dots}\cup\F_n$.
\end{lemma}
\begin{proof}
Take $f\in\F_i$, and consider the map $\tilde{f}:X^m\to X$ by
$
\tilde{f}(x_1,{\dots},x_m):=f(\pi_i(x_1),{\dots},\pi_i(x_m)).
$
Then, obviously, $\tilde{f}$ is an extension of $f$ and $\on{Lip}(\tilde{f})\leq \on{Lip}(\pi_i)\cdot\on{Lip}(f)$. Then it is enough to take $\mH=\{\tilde{f}:f\in\F_1\cup{\dots}\cup\F_n\}$.
\end{proof}

\begin{remark}\label{nr1}{Observe that condition (b) holds if
\begin{itemize}
\item[(c)] $\lambda:=\frac{\max\{\on{diam}(X_i):i=1,{\dots},n\}}{\min\{\on{dist}(X_i,X_j):i,j=1,{\dots},n,\;i\neq j\}}<\frac{1}{\max\{\on{Lip}(\F_i):i=1,{\dots},n\}}$.
\end{itemize}
Indeed, let the projection maps $\pi_i$ be given by $\pi_i(x):=\left\{\begin{array}{ccc}x,&\mbox{if } x\in X_i\\\tilde{x}&\mbox{otherwise}\end{array}\right.$, where $\tilde{x}$ is an initially chosen point of $X_i$. Then for $x\in X_i$ and $y\in X_{j}$, where $i\neq j$, we have
$$
d(\pi_i(x),\pi_i(y))\leq \on{diam}(X_i)\leq\lambda\on{dist}(X_i,X_j)\leq \lambda d(x,y),
$$
so $\on{Lip}(\pi_i)\leq\max\{1,\lambda\}<\frac{1}{\max\{\on{Lip}(\F_i):i=1,{\dots},n\}}$ and (b) is satisfied.
}\end{remark}

The following lemma is an extension of \cite[Lemma 1]{N}.
\begin{lemma}\label{nl3}
Assume that $n\in\N$ and a metric space $X$ is of the form $X=X_0\cup{\dots}\cup X_n$, where: 
\begin{itemize}
\item[(i)] each $X_i$ is isometric copy of $X_0$;
\item[(ii)] $\on{dist}(X_i,X_j)\geq \on{diam}(X_0)$ for distinct $i,j$.
\end{itemize}
If $X$ is the attractor of a weak GIFS of order $m$, then $X_0$ is the attractor of a~weak GIFS of order $m$.
\end{lemma}
\begin{proof}
For $i=0,1,{\dots},n$, let $t_i:X_0\to X_i$ be an isometry. Suppose that $\F$ is a GIFS of order $m$ such that $X=\bigcup_{f\in\F}f(X^m)$. For every $f\in\F$ and $\mathbf{i}=(i_1,{\dots},i_m)\in\{0,{\dots},n\}^m$, let $H_{f,\mathbf{i}}:X_0^m\to X_0$ be defined by
$$
H_{f,\mathbf{i}}(x_1,{\dots},x_m)=\left\{\begin{array}{ccc}f(t_{i_1}(x_1),{\dots},t_{i_m}(x_m))&\mbox{if}&f(t_{i_1}(x_1),{\dots},t_{i_m}(x_m))\in X_0\\
z_0&\mbox{if}&f(t_{i_1}(x_1),{\dots},t_{i_m}(x_m))\notin X_0\end{array}\right.
$$
where $z_0$ is an initially chosen point of $X_0$.
Clearly,
$
\bigcup_{f\in\F,\;\mathbf{i}\in\{0,{\dots},n\}^m}H_{f,\mathbf{i}}(X_0^m)=X_0
$. We will show that each $\on{Lip}(H_{f,\mathbf{i}})\leq\on{Lip}(f)$. Take $x=(x_1,{\dots},x_m),\;y=(y_1,{\dots},y_m)\in X_0^m$ such that $H_{f,i}(x)\neq H_{f,i}(y)$.
\newline 
If $f(t_{i_1}(x_1),{\dots},t_{i_m}(x_m)),f(t_{i_1}(y_1),{\dots},t_{i_m}(y_m))$ belong to distinct $X_i,X_j$, then by assumption (ii), we have
$$
d(H_{f,\mathbf{i}}(x),H_{f,\mathbf{i}}(y))\leq\on{diam}(X_0)\leq$$ $$\leq \on{dist}(X_i,X_j)\leq d\big(f(t_{i_1}(x_1),{\dots},t_{i_m}(x_m)),f(t_{i_1}(y_1),{\dots},t_{i_m}(y_m))\big)<$$ $$<d_m\big((t_{i_1}(x_1),{\dots},t_{i_m}(x_m)),(t_{i_1}(y_1),{\dots},t_{i_m}(y_m))\big)=d_m(x,y).
$$
If $f(t_{i_1}(x_1),{\dots},t_{i_m}(x_m)),f(t_{i_1}(y_1),{\dots},t_{i_m}(y_m))\in X_0$ then
$$
d(H_{f,\mathbf{i}}(x),H_{f,\mathbf{i}}(y))=d\big(f(t_{i_1}(x_1),{\dots},t_{i_m}(x_m)),f(t_{i_1}(y_1),{\dots},t_{i_m}(y_m))\big)<$$ $$<d_m\big((t_{i_1}(x_1),{\dots},t_{i_m}(x_m)),(t_{i_1}(y_1),{\dots},t_{i_m}(y_m))\big)=d_m(x,y).
$$
In {remaining cases} we proceed in a trivial way.
\end{proof}

\subsection{Quotients of topological fractals}
We first show that if $X$ is a~topological fractal, then also its certain quotient space is a topological fractal (in fact, this result was proved during preparations of the paper \cite{BKNNS}, but it was not published). We start with recalling some basic facts about quotient spaces (see \cite{En}). Let $X$ be a topological space and $R$ be an equivalence relation on $X$. By $[x]$ we denote the equivalence class containing $x$, by $X/R$, the set of all equivalence classes, and by $\pi:X\to X/R$, the map $\pi(x):=[x]$. If $\tau$ is the topology on $X$, then the quotient topology on $X/R$ is defined by
$
\tau_R:=\{V\subset X/R:\pi^{-1}(V)\in\tau\}=\{V\subset X/R:\bigcup V\in\tau\}
$ 
that is, $\tau_R$ is the richest topology so that $\pi$ is continuous.\\ 
If $x\in X$, then by $[x]_c$ let us denote the connected component containing $x$, that is, the union of all connected subsets of $X$ which contains $x$. It is well known (again, see \cite{En}) that two different connected components are disjoint, hence the relation $R_c$ defined by $x R_c y$ iff $x,y$ belong to the same connected component, is equivalence relation. Moreover, each connected component is connected and closed. 
\begin{theorem}\label{quotient}
If $X$ is {the} topological fractal for an IFS $\F$, then $X/R_c$ is {the} topological fractal for the IFS $\tilde{\F}=\{\tilde{f}:f\in\F\}$, where for each $f\in\F$ and $x\in X$,
$
\tilde{f}([x]_c):=[f(x)]_c
$.
\end{theorem}
\begin{proof}
As the image of connected set via a continuous function must be connected, and connected components are connected, the map $\tilde{f}:X/R_c\to X/R_c$ defined in thesis
is well defined. We will show that it is continuous. Let $V$ be open in $X/R_c$. {Then $\pi^{-1}(\tilde{f}^{-1}(U))=f^{-1}(\pi^{-1}(U))$ is open as $f$ and $\pi$ are continuous. Hence $\tilde{f}^{-1}(U)$ is open in $X/R_c$ and $\tilde{f}$ is continuous.} 
Now observe that
\begin{equation}\label{filipnowy81}
\bigcup_{\tilde{f}\in\tilde{\F}} \tilde{f}(X/R_c)=X/R_c
\end{equation}
Take $[y]_c\in X/R_c$. Then for some $x\in X$ and $f\in\F$, $f(x)=y$, so $\tilde{f}([x]_c)=[f(x)]_c=[y]_c$. We get (\ref{filipnowy81}).\\
Finally, take any sequence $(\tilde{f}_n)\subset \tilde{\F}$. We will show that $\bigcap_{n\in\N}\tilde{f}_1\circ{\dots}\circ \tilde{f}_n(X/R_c)$ is singleton. Let $[x]_c$ be any element of this intersection. Then for every $n\in\N$, there exists $x_n\in X$ such that 
$$
[x]_c=\tilde{f}_1\circ{\dots}\circ \tilde{f}_n([x_n]_c)=\tilde{f}_1\circ{\dots}\circ\tilde{f}_{n-1}([f_n(x_n)]_c)={\dots}=[f_1\circ{\dots}\circ f_n(x_n)]_c.
$$
In particular, $f_1\circ{\dots}\circ f_n(x_n)\in [x]_c$. In particular, $f_1\circ{\dots}\circ f_n(x_n)\in f_1\circ{\dots}\circ f_n(X)$.
Now for $n\in\N$, let $y_n:=f_1\circ{\dots}\circ f_n(x_n)$. Since $X$ is topological fractal, it is compact metrizable and by the first observation, $(y_n)$~has a~convergent subsequence whose limit $y\in[x]_c$ (as $[x]_c$ is closed). But by the second one, this limit must belong to $\bigcap_{n\in\N}f_1\circ{\dots}\circ f_n(X)$. Hence if $[x']_c$ is different from $[x]_c$ and contained in $\bigcap_{n\in\N}\tilde{f}_1\circ{\dots}\circ \tilde{f}_n(X/R_c)$, we would also find $y'\in[x']_c$ contained in $\bigcap_{n\in\N}f_1\circ{\dots}\circ f_n(X)$. As $y\neq y'$, this would be a contradiction.\\
{Finally, $X/R_c$ is compact and Hausdorff - see for example \cite[Chapter V.\S 46.Va]{Kuratowski}.}
\end{proof}

\section{$(\La,b)$-spaces}
If $\xi=(\xi_1,{\dots},\xi_k)$, then we {define} \emph{the length of} $\xi$ by $|\xi|:=k$, if $\xi=(\xi_1,\xi_2,{\dots})$, then we set $|\xi|:=\omega$, and also we {put} $|\emptyset|:=0$. If $i\leq |\xi|$, then we set $[\xi]_i:=(\xi_1,{\dots},\xi_i)$ (and also $[\xi]_0:=\emptyset$). If $\xi,\eta$ are sequences, then we write $\xi\preceq \eta$, if $[\eta]_k=\xi$ for some $k\leq|\eta|$. Finally, if $\xi,\eta$ are sequences such that $|\xi|<\omega$, then by $\xi\ha \eta$ we denote the concatenation of $\xi$ and $\eta$.\\
If $A$ is a set and $k\in\N\cup\{0\}$, then by $A^\omega$, $A^{<\omega}$ and $A^k$ we denote the families of all countable infinite, finite and of the length $k$ sequences of elements of $A$, respectively.\\ 
Let $\oN:=\N\cup\{\omega\}=\{\omega,1,2,{\dots}\}$. We say that a nonempty set $\Lambda\subset \oN^{<\omega}$ is \emph{a tree}, if for every $\xi\in\Lambda$ and $k\leq |\xi|$, also $[\xi]_k\in \Lambda$. If $\Lambda$ is a tree, then the \emph{boundary} of $\Lambda$ is defined by
$$
\tL:=\{\xi\in \La:\forall_{k\in\oN}\;\xi\ha k\notin \La\}\cup\{\xi\in \oN^\omega:\forall_{k\in\N}\;[\xi]_k\in\La\}.
$$
Observe that $\tL$ can be identified with a family of all maximal $\preceq$-linearly ordered subsets of $\La$. Moreover, $\tL\cap \La$ consists of all finite sequences from $\tL$, and $\tL\setminus \La$ - of all infinite sequences from $\tL$.\\
We say that a tree $\La$ is \emph{proper}, if
\begin{itemize}
\item[(pi)] for every $\xi\in \La$ which ends with $\omega$, it holds $\xi\in \tL$;
\item[(pii)] for every $\xi\in\La\setminus \tL$ and $k\in\oN$, $\xi\ha k\in\La$.
\end{itemize} 
By $\La_{\max}$ we will mean the maximal proper tree, that is, the tree which consists of the empty sequence and all sequences $\eta=(\eta_1,{\dots},\eta_k)$ so that $\eta_i\in\N$ for $i=1,{\dots},k-1$. In this case, the boundary $\tL_{\max}$ consists of all finite sequences which ends with $\omega$ and all infinite sequences of elements of $\N$.\\We say that a sequence $b=(b_i)_{0\leq i\leq \omega}$ of positive reals is \emph{good}, if $b_\omega=0$ and
\begin{equation}\label{181}
M_b:=\sup\left\{\frac{b_{k}}{b_{k-1}}:k\in\N\right\}<\frac{1}{25}.
\end{equation}
It is easy to see that in this case
\begin{equation}\label{181a}
\lambda_b:=25M_b<1
\end{equation}
and for every $k\in\N$,
\begin{equation}\label{181b}
b_k\leq \frac{1}{20}\lambda_b(b_{k-1}-2b_k-b_{k+1}).
\end{equation}
The above conditions looks artificial but, as we will see, they will be needed at some places later (even more technical assumptions will be made in Section 5). On the other hand, restricting to good sequences will not decrease the generality of our constructions.\\
Finally, for every $i\in\N\cup \{0\}$, define $b^i:=(b_i,b_{i+1},{\dots})$. Clearly, $b^i$ is a good sequence, $M_{b^i}\leq M_b$ and $\lambda_{b^i}\leq\lambda_b$.\\
For a sequence $\eta\in \oN^{\leq \omega}$, we set $l(\eta)$ by (we assume $n+\omega=\omega$)
\begin{equation}\label{182}
l(\eta):=\left\{\begin{array}{ccc}\eta_1+{\dots}+\eta_k&\mbox{if}&\eta=(\eta_1,{\dots},\eta_k),\\
\omega&\mbox{if}&\eta\in\oN^{\omega},\\
0&\mbox{if}&\eta=\emptyset.
\end{array}\right.
\end{equation}
Now, relying on the above notations, we will define a certain class of metric spaces. All our next consideration will base on it.

\begin{definition}\label{183}\emph{
Let $\La$ be a proper tree and $b$ be a good sequence. A compact metric space $X$ will be called a }$(\La,b)$-space, \emph{if $X$ is of the form $X=\bigcup_{\eta\in \La\cup\tL}X_\eta$, where $X_\emptyset=X$ and:\\
 for every $\eta\in\tL$,
\begin{itemize}
\item[(i)] $X_\eta$ is nonempty, compact and contained in $X_\xi$ for every $\xi\preceq \eta$;
\end{itemize}
for every
 $\eta\in \La\setminus \tL$,
\begin{itemize}
\item[(ii)] $X_\eta=\bigcup_{k\in\oN}X_{\eta\ha k}$;
\item[(iii)] for every $k\in \N$, $\on{dist}{\big(}X_{\eta\ha k},\bigcup_{k<j\leq\omega}X_{\eta\ha j}{\big)}\geq b_{l(\eta)+k-1}-2b_{l(\eta)+k}-b_{l(\eta)+k+1}$;
\item[(iv)] $\on{diam}{\big(}\bigcup_{k\leq j<\omega}X_{\eta\ha j}{\big)}\leq b_{l(\eta)+k-1}$;
\item[(v)] there is $x_{\eta\ha\omega}\in X_{\eta\ha \omega}$ such that for every $k\in\N$, $x_{\eta\ha\omega}\in\overline{\bigcup_{k\leq j<\omega}X_{\eta\ha j}}$.
\end{itemize}
If $X$ is a $(\La,b)$-space such that for every $\eta\in  \tL$,
\begin{itemize}
\item[(s)] $X_\eta$ is singleton,
\end{itemize}
then we call $X$ a }$(\La,b,s)$-space\emph{. In such case, if $\eta\in\tL$ then the unique element of $X_\eta$ we denote by~$x_\eta$ (clearly, there is no collision with (v)).\\
If $X$ is a $(\La,b)$-space such that for some compact metric space $Z$ and every $\eta,\xi\in \La\cap \tL$,
\begin{itemize}
\item[(Z)] $X_\eta$ is a {similarity} copy (that is, {an image under similitude, i.e., isometry} with a scale) of $Z$ such that
\begin{itemize}
\item[(Z1)] if $\eta$ does not end with $\omega$, then $\on{diam}(X_\eta)=b_{l(\eta)}+b_{l(\eta)+1}$;
\item[(Z2)] if $\eta$ ends with $\omega$, then $\on{diam}(X_\eta)=b_{l(\tilde{\eta})+1}$, where $\tilde{\eta}$ is such that $\eta=\tilde{\eta}\ha\omega$;
\item[(Z3)] if $\eta$ and $\xi$ end with $\omega$, then there is a {similitude} $h:X_\eta\to X_\xi$ such that $h(x_\eta)=x_\xi$.
\end{itemize}
\end{itemize}
then we call $X$ a }$(\La,b,Z)$-space.
\end{definition}
\noindent Observe that conditions (i) and (iv) imply that $X_\eta$ is singleton for $\eta \in \tL\setminus\Lambda$.

If $X$ is a $(\La,b)$-space, $\eta\in\La\setminus\tL$ and $k\in\N$, then we also set
$$
\tilde{X}_{\eta\ha k}:=\bigcup_{k\leq j\leq \omega}X_{\eta\ha j}.
$$
Observe that if $\eta\in\La\setminus \tL$, then $X_\eta=\tilde{X}_{\eta\ha 1}$.\\
We first list basic properties of $(\La,b)$-spaces (by $d$ we will denote the metric on a metric space $X$):
\begin{lemma}\label{1822}
Let $X$ be a $(\La,b)$-space.
\begin{itemize}
\item[{(a)}] For every $\eta\in\La$ which does not end with $\omega$, $X_{\eta}$ and $\tilde{X}_{\eta}$ are compact and open;
\item[{(b)}] if $\eta\in\La\setminus \tL$ and $x\in X_{\eta\ha\omega},\;y\in X_\eta\setminus X_{\eta\ha\omega}$, then $d(y,x_{\eta\ha \omega})\leq 2d(x,y)$ and $d(x,x_{\eta \ha \omega})\leq 3d(x,y)$.
\end{itemize}
If $X$ is a $(\La,b,s)$-space which is not singleton, then
\begin{itemize}
\item[{(c)}] the family $$\{X_\eta:\eta\in\La,\;\eta\mbox{ does not end with }\omega\}\cup\{\tilde{X}_\eta:\eta\in\La,\;\eta\mbox{ does not end with }\omega\}$$
is a basis of $X$ consisting of clopen sets;
\item[{(d)}] if $\eta\in \La$, then $\La_1[\eta]:=\{\beta:\eta\ha\beta\in \La\}$ is proper tree and $X_\eta$ is $(\La_1[\eta],b^{l(\eta)},s)$-space;
\item[{(e)}] if $\eta\in\Lambda$ and $k\in\N$ are such that $\eta \ha k \in \La$, then $\La_2[\eta\ha k]:=\{i\ha\beta:i\in\N,\; \eta\ha(i+k-1)\ha\beta\in\La\}\cup\{\emptyset,\omega\}$ is proper tree, and $\tilde{X}_{\eta\ha k}$ is $(\La_2[\eta\ha k],b^{l(\eta\ha k)-1},s)$-space;
\item[{(f)}] if $Y$ is a $(\La,b',s)$-space for a good sequence $b'$, then $X$ and $Y$ are homeomorphic.\\
\end{itemize}
\end{lemma}
\begin{proof} 
We first prove {(a)}. Let $\xi \in\La$ does not end with $\omega$. $X_\xi$ and $\tilde{X}_\xi$ are open as, by (ii), (iii) and (\ref{181b}), $\on{dist}(X_\xi,X\setminus X_\xi)\geq b_{l(\xi)-1}-2b_{l(\xi)}-b_{l(\xi)+1}>0$ and $\on{dist}(\tilde{X}_\xi,X\setminus \tilde{X}_\xi)\geq b_{l(\xi)-2}-2b_{l(\xi)-1}-b_{l(\xi)}>0$. Now~we show that they are compact. We first prove that $X_\xi$ is compact. Suppose $\xi\notin \tL$, choose a sequence $(x_n)\subset X_\xi$ and consider the following cases:\\
Case 1. Infinitely many elements $x_n$ belong to $X_{\xi\ha\eta}$ for some $\eta$ with $\xi\ha\eta\in \tL$. Then $(x_n)$ has convergent subsequence by (i).\\
Case 2. Some subsequence of $(x_n)$ converges to some element of $X_{\xi\ha\eta}$, where $\eta$ ends with $\omega$. Then we are done.\\
Case 3. The previous cases do not hold.\\
Step by step, we will define a sequence $\eta=(\eta_1,\eta_2,{\dots})$ such that $\xi\ha\eta\in\tL$ and for every $k$, infinitely many elements $x_n$ belong to $X_{\xi\ha[\eta]_k}$. This will allow to choose a subsequence $(x_{n_k})$ such that $x_{n_k}\in X_{\xi\ha[\eta]_k}$ for $k\in\N$. By (iv) and a fact that $\emptyset\neq X_{\xi\ha\eta}\subset X_{\xi\ha[\eta]_k}$ for every $k\in\N$, this will mean that $(x_{n_k})$ is convergent to the unique element of $X_{\xi\ha\eta}\subset X_\xi$.\\
Now we show how to choose $\eta$. By our assumptions, (iv) and (v), there is $k\in\N$ such that infinitely many elements $x_n$ belong to $\bigcup_{j\leq k_1}X_{\xi\ha j}$. Hence we can find $\eta_1\in\{1,{\dots},k_1\}$ such that infinitely many elements of $x_n$ belong to $X_{\xi\ha\eta_1}$. Similarly, there is $k_2\in\N$ so that infinitely many elements of $x_n$ belong to $\bigcup_{j\leq k_2}X_{\xi\ha \eta_1\ha j}$, and hence we find $\eta_2\in\{1,{\dots},k_2\}$ so that infinitely many elements $x_n$ belong to $X_{\xi\ha\eta_1\ha\eta_2}$. We can continue this procedure as we assume that Cases 1 and 2 do not hold.\\
We proved that $X_\xi$ is compact. Now let $\tilde{\xi}$ be such that $\xi=\tilde{\xi}\ha k$ for some $k\in\N$. Then $\tilde{X}_{\xi}=X_{\tilde{\xi}}\setminus\bigcup_{i<k}X_{\xi\ha i}$ and, by what we already proved, $\tilde{X}_{\xi}$ is compact.

Now we prove {(b)}. Assume that $x\in X_{\eta\ha\omega}$ and $y\in X_{\eta\ha k}$ for some $k\in\N$. By conditions (iii)-(v) and (\ref{181b}), we have
$$
d(x_{\eta\ha \omega},y)\leq \on{diam}{\Bigg(}\overline{\bigcup_{k\leq j<\omega}X_{\eta\ha j}}{\Bigg)} =\on{diam}{\Bigg(}\bigcup_{k\leq j<\omega}X_{\eta\ha j}{\Bigg)}\leq b_{l(\eta)+k-1}
$$
and
$$
d(x,y)\geq \on{dist}{\bigg(}X_{\eta\ha k},\bigcup_{k<j\leq\omega}X_{\eta\ha j}{\bigg)}\geq b_{l(\eta)+k-1}-2b_{l(\eta)+k}-b_{l(\eta)+k+1}\geq  $$
$$\geq b_{l(\eta)+k-1}-\frac{5}{2}b_{l(\eta)+k}\geq b_{l(\eta)+k-1}{\bigg(}1-\frac{5}{2}M_b{\bigg)}\geq 
\frac{1}{2}b_{l(\eta)+k-1}\geq \frac{1}{2}d(x_{\eta\ha\omega},y).
$$
Hence we get the first inequality. To see the second, use the first one:
$$
d(x,x_{\eta\ha\omega})\leq d(x,y)+d(y,x_{\eta\ha\omega})\leq 3d(x,y).
$$
Point {(c)} follows from {(a)}, (ii) and facts that $\on{diam}(X_\eta)\leq b_{l(\eta)}$ and $\on{diam}(\tilde{X}_{l(\eta)})\leq b_{l(\eta)-1}$ (which are implied by (iii) and (iv)). \\
To see {(d)}, it is routine to check that $\La_1[\eta]$ is proper and the family $X_\beta':=X_{\eta\ha\beta},\;\beta\in\La_1[\eta]$, satisfies the required conditions. \\
Similarly, to get {(e)}, it is easy to see that $\La_2[\eta\ha k]$ is proper and that the family $X''_{i\ha\beta}:=X_{\eta\ha(i+k-1)\ha\beta}$, $i\ha\beta\in\La_2[\eta\ha k]$, satisfies the required conditions. (in fact, {(d)} follows from {(e)} by our earlier observation).\\ 
To see {(f)}, note that the map $X\ni x_\eta\to y_\eta\in Y$ is homeomorphism, which can be easily proved using~{(c)}.

\end{proof}

Now we show that the above definition is non void, and we can embed $(\La,b)$ spaces into $\R$ or other Hilbert spaces. We should start with defining appropriate "skeleton" on the real line.\\
Let $b$ be a good sequence. We say that a family $\I=\{I_\xi:\xi\in \Lambda_{\on{max}}\}$ of closed intervals is a \emph{$b$-family}, if for every $\eta\in \Lambda_{\on{max}}$ which does not end with $\omega$,
\begin{itemize}
\item[($\I$1)] $\on{diam}(I_\eta)=b_{l(\eta)}+b_{l(\eta)+1}$;
\item[($\I$2)] for every $k\in\oN$,
\begin{itemize}
\item[($\I$2a)] if $k\in\N$, then $\on{max}I_{\eta\ha k}=\min I_\eta+b_{l(\eta)+1}+b_{l(\eta)+k-1}$;
\item[($\I$2b)] if $k=\omega$, then $\min I_{\eta\ha\omega}=\min I_\eta$, and $\on{max}I_{\eta\ha \omega}=\min I_\eta+b_{l(\eta)+1}$;
\end{itemize}
\end{itemize}
Observe that $I_{\eta \ha k} \subset I_\eta$, $\on{max} I_\eta = \on{max} I_{\eta \ha 1}$ for $\eta$ which does not end with $\omega$, and $\on{dist}(I_{\eta\ha k},I_{\eta\ha (k+1)})=b_{l(\eta)-1}-2b_{l(\eta)}-b_{l(\eta)+1}$.
Additionally, for every infinite $\eta\in\tL_{\max}$, set $I_\eta:=\bigcap_{k\in\N}I_{[\eta]_k}$ (clearly, it is a singleton).
\begin{remark}{It is worth to observe that we can look on the family $\{I_\eta:\eta\in\La_{\max},\;\eta\mbox{ does not end with }\omega\}$ as on standard ternary Cantor scheme on $\R$ (with the length of each interval of the $k$-th step equal to $b_{k}$), but for our purposes we enumerate these intervals in a different way.}
\end{remark}

\begin{theorem}\label{fvf3}
Assume that $\La$ is a proper tree and $b$ is a good sequence.\\
(1) There exists $X\subset \R$ which is a $(\La,b,s)$-space.\\
(2) If $Z$ is a nonempty, non singleton, compact subset of a {normed} space $H$, then there exists $X\subset H$ which is a $(\La,b,Z)$-space.\\
(3) If $Z$ is a nonempty compact metric space which is not a singleton, then there exists a metric space $X$ which is a $(\La,b,Z)$-space.
\end{theorem}
\begin{proof}
We first prove (1). Take a $b$-family $\I$ and $\eta\in\tL$. If $\eta$ ends with $\omega$~then set $x_\eta := \on{max} I_\eta$, otherwise let $x_\eta$ be any element of $I_\eta$.
Then it is easy to see that $X:=\{x_\eta:\eta\in\tL\}$ is a $(\La,b,s)$-space (for the family $X_\eta:=\{x_\xi:\xi\in \tL,\;\eta\preceq \xi\},\;\eta\in\La$).

Now we prove (2). We first {state} the following Claim:\\
\textbf{Claim.} {For any distinct points $x,y$ of a normed space $H$ and any distinct points $x',y'\in \R(x-y):=\{t(x-y):t\in\R\}$ there exists a bijective {similitude} $h:H\to H$ such that $h(x)=x'$ and $h(y)=y'$.\\
The {similitude} $h$ can be defined by $h(z):=x'+\frac{||y'-x'||}{||y-x||}(z-x)$ or $h(z):=x'-\frac{||y'-x'||}{||y-x||}(z-x)$, depending on the mutual relationship between $x'$ and $y'$.\\}
Since $Z$ is compact, there are $x_0,y_0\in Z$ such that $||x_0-y_0||=\on{diam}(Z)$.
Now choose any line $l\subset H$, and identify $l$ with the real line $\R$, and find any $b$-family $\I$ on $l$. By the Claim, for every $\eta\in\La\cap \tL$, we~can find a {similitude} $h_\eta:H\to H$ such that $h_\eta(x_0)=\min I_\eta$ and $h_\eta(y_0)=\max I_\eta$. Set $X_\eta:=h_\eta(Z)$. Moreover, if $\eta\in\tL\setminus \La$, then let $X_\eta:=I_\eta=\bigcap_{k\in\N}I_{[\eta]_k}$ (clearly, $X_\eta$ is singleton). For every $\eta\in\La$, let~$X_\eta:=\bigcup_{\xi\in \tL,\;\eta\preceq \xi}X_\xi$, and finally set $X:=X_\emptyset = \bigcup_{\eta\in\tL}X_\eta$. We will prove that $X$ is a $(\La,b,Z)$-space.\\
Clearly, the condition (Z) is satisfied (as $\on{diam}(h(Z))=||h(x_0)-h(y_0)||$ for every considered {similitude} and, if $\eta$ ends with $\omega$, then $x_\eta$ is the image of $y_0$ via $h$). Condition (i) follows from the fact that $Z$ is nonempty and compact, and (ii) directly from definition.\\
Now we show that for every $\eta\in\La\cup\tL$,
\begin{equation}\label{fvf1}
\min I_\eta,\max I_\eta\in X_\eta\;\;\;\mbox{and}\;\;\;\on{diam}(X_\eta)=\on{diam}(I_\eta).
\end{equation}
If $\eta\in\tL$, then (\ref{fvf1}) clearly holds. Hence assume $\eta\in\La\setminus\tL$. By definition, $\min I_\eta=\min I_{\eta\ha\omega}\in X_{\eta\ha\omega}\subset X_\eta$. Now if $\eta\ha(1,1,1,{\dots})\in\tL$, then $\max I_\eta\in\bigcap_{k\in\N} I_{\eta\ha[(1,1,{\dots})]_k}=X_{\eta\ha(1,1,{\dots})}\subset X_\eta$, and if for some $k\in\N$, $\eta\ha[(1,1,{\dots})]_k\in\tL$, then $\max I_\eta=\max I_{\eta\ha[(1,1,{\dots})]_k}\in X_{\eta\ha[(1,1,{\dots})]_k}\subset X_\eta$.\\
Now choose any $x,y\in X_\eta$ and let $\xi,\xi'\in\tL$ be such that $\eta\preceq \xi$, $\eta\preceq \xi'$, $x\in X_{\xi}$ and $y\in X_{\xi'}$. If $\xi=\xi'$, then $||x-y||\leq\on{diam}(I_\xi)\leq\on{diam}(I_\eta)$. Hence assume $\xi\neq  \xi'$ and, without loss of generality, that {$\xi_{k}>\xi'_{k}$ for some $k>|\eta|$}. Then
$$
||x-y||\leq ||x-\max I_{\xi}||+||\max I_{\xi}-\min I_{\xi'}||+||\min I_{\xi'}-y||\leq $$ $$\leq ||\min I_\xi-\max I_{\xi}||+||\max I_{\xi}-\min I_{\xi'}||+||\min I_{\xi'}-\max I_{\xi'}||=$$ $$=||\min I_\xi-\max I_{\xi'}||\leq \on{diam}(I_\eta).
$$
This ends the proof of (\ref{fvf1}). Now we prove (iii), (iv) and (v). We will use (\ref{fvf1}) and definition of a~$b$-family. Take any $\eta\in\La\setminus\tL$ and $k<j\leq\omega$. Then for every $x\in X_{\eta\ha k}$ and $y\in X_{\eta\ha j}$, we have by (\ref{fvf1}),
$$
||\max I_{\eta\ha k}-\min I_{\eta\ha k}||+||\min I_{\eta\ha k}-\max I_{\eta\ha j}||+||\max I_{\eta\ha j}-\min I_{\eta\ha j}||=$$ $$=||\max I_{\eta\ha k}-\min I_{\eta\ha j}||\leq ||\max I_{\eta\ha k}-x||+||x-y||+||y-\min I_{\eta\ha j}||\leq $$ $$\leq ||\max I_{\eta\ha k}-\min I_{\eta\ha k}||+||x-y||+||\max I_{\eta\ha j}-\min I_{\eta\ha j}||
$$
so
$$
||x-y||\geq ||\min I_{\eta\ha k}-\max I_{\eta\ha j}||\geq b_{l(\eta)+k-1}-2b_{l(\eta)+k}-b_{l(\eta)+k+1}.
$$
We proved (iii). Now assume additionally $j<\infty$. Then
$$
||x-y||\leq||x-\min I_{\eta\ha k}||+||\min I_{\eta\ha k}-\max I_{\eta\ha j}||+||\max I_{\eta\ha j}-y||\leq $$ $$\leq ||\max I_{\eta\ha k}-\min I_{\eta\ha k}||+||\min I_{\eta\ha k}-\max I_{\eta\ha j}||+||\max I_{\eta\ha j}-\min I_{\eta\ha j}||\leq$$ $$\leq \on{diam}{\Bigg(}\bigcup_{k\leq i<\omega}I_{\eta\ha i}{\Bigg)}\leq b_{l(\eta)+k-1}
$$
and we proved (iv). Finally, (\ref{fvf1}) implies that $x_{\eta\ha\omega}=\max I_{\eta\ha\omega}\in \overline{\bigcup_{k\leq i<\omega}X_{\eta\ha j}}$, hence (v) also holds.
This ends the proof of (2).

We just sketch the proof of (3). The idea is similar as in point (2) - we choose a $b$-family $\I$ and try to replace $I_\eta$ by appropriate copy of $Z$, for $\eta\in\tL$.\\
For every $\eta\in\tL\cap\La$, let $X_\eta$ be a {similarity} copy of $Z$ according to $(Z1),(Z2)$ and let $d_\eta$ be initial metric on $X_\eta$. If $\eta\in\tL$ ends with $\omega$, then let $x_\eta\in X_\eta$ be the image, via the {similitude}, of initially chosen point $y_0\in Z$ such that for some $x\in Z$, $\rho(x,y_0)=\on{diam}(Z)$. Thanks to it, the condition (Z3) will be satisfied. 
If $\eta\in\tL$ is infinite, then let $X_\eta$ be singleton. 
Moreover, for every $\eta\in\La\setminus\tL$, let $X_\eta$ be defined like before in the proof of (2), and finally set $X:=X_\emptyset$. Now we will define appropriate metric $d$ on $X$.\\ For every $\xi\in\La$, let
$$
z_\xi:=\left\{\begin{array}{ccc}\mbox{midpoint of }I_\xi&\mbox{if}&\xi\mbox{ does not end with }\omega\\
\max I_\xi&\mbox{if}&\xi\mbox{ ends with }\omega\end{array}\right.,
$$
If $\eta\in\tL\cap\La$, then let the metric $d$ coincide with $d_\eta$ on $X_\eta$. Now let $x,y\in X$ be distinct and assume that they do not belong to the same $X_\eta$ for $\eta\in\tL$. Without loss of generality, there exist $\eta\in\La\setminus\tL$ and $k<j\leq\omega$, such that $x\in X_{\eta\ha k},\;y\in X_{\eta\ha j}$. If $j\neq\omega$, then define
$$
d(x,y):=|z_{\eta\ha k}-z_{\eta\ha j}|,
$$
and if $j=\omega$, then define
$$
d(x,y):=d(y,x_{\eta\ha \omega})+|z_{\eta\ha \omega}-z_{\eta\ha k}|.
$$
It is routine to check that $d$ is indeed a metric and that $(X,d)$ is a $(\La,b,Z)$-space.

\end{proof}

We will also need to deal with certain subsets of $(\La,b,s)$-spaces, hence we introduce the following notation: given a $(\La,b,s)$-space $X$ and its compact subset $M\subset X$, we set $\La_M:=\{\eta\in\La:X_\eta\cap M\neq\emptyset\}$. We skip the proof of the following observation:
\begin{observation}\label{1819}
In the above frame:
\begin{itemize}
\item[(M1)] $\La_M$ is a tree (but need not be proper);
\item[(M2)] $\tL_M=\{\eta\in\tL:x_\eta\in M\}$;
\item[(M3)] $M=\{x_\eta:\eta\in\tL_M\}$;
\item[(M4)] If $N\neq M$ and $N$ is also compact, then $\La_N\neq\La_M$.
\end{itemize}
\end{observation}

We end this section with two general Lemmas, which will be keys in later constructions of appropriate GIFSs.

\begin{lemma}\label{1814}
Assume that $X$ is a $(\La,b)$-space, $\emptyset\neq Z\subset X$ and 
 for every $2\leq k<\omega$, there exists a map $h_k:Z\to X$ so that for every $2\leq k<j<\omega$,
\begin{itemize}
\item[(i)] $\on{Lip}(h_k)\leq \frac{1}{4} \lambda_b$;
\item[(ii)] for all $x\in Z$, $d(h_k(x),h_j(x))\leq 2b_{k-1}$.
\end{itemize}
Let $h_\omega (x) := \lim_{k\to \infty} h_k(x), x\in X$, and for $k\in\oN$, let $Z_k:=Z\cap X_k$. Then the map $F:Z\times Z\to X$ defined by (we set $\omega+1=\omega$) $$F(x,y):=h_{k+1}(x),\mbox{ if }y\in Z_k,\mbox{ }k\in\oN,$$
satisfies the following
\begin{equation*}
\on{Lip}(F)\leq \frac{1}{2}\lambda_b\;\;\mbox{ and }\;\;F(Z\times Z)=\bigcup\{h_k(Z):k\in\oN,\;k\geq 2,\;Z_{k-1}\neq\emptyset\}.
\end{equation*}
\end{lemma}
\begin{proof}
Since $X$ is compact and $(h_k(x))_{k\in\N}$ is Cauchy, the map $h_\omega$ is well defined. By definition, we have
$$F(Z \times Z) =  F{\Bigg(}Z \times \bigcup_{k\in\oN,\;Z_k\neq\emptyset} Z_k{\Bigg)}  =\bigcup\{h_k(Z):k\in\oN,\;k\geq 2,\;Z_{k-1}\neq\emptyset\}.$$
We will show that $\on{Lip}(F)\leq \frac{1}{2}\lambda_b$.
Take distinct $(x,y), (x',y') \in Z \times Z$ so that $F(x,y)\neq F(x',y')$ and~consider cases.\\
Case 1. {For some $k\in\oN$, it holds $y,y' \in Z_k$}. Then
$$
d(F(x,y), F(x',y')) = d(h_{k+1}(x) , h_{k+1}(x')) \leq $$ $$\leq \on{Lip}(h_{k+1}) d(x,x') \leq \frac{1}{4}\lambda_b d_m((x,y), (x',y')).
$$
Case 2. For some $k<j\leq\omega${, it holds} $y\in Z_k$ and $y' \in Z_j$. Then by Case 1, (iii) and (\ref{181b}),
$$
d(F(x,y),F(x',y'))\leq d(F(x,y),F(x,y'))+d(F(x,y'),F(x',y'))\leq $$ $$\leq d(h_{k+1}(x),h_{j+1}(x))+\frac{1}{4}\lambda_bd(x,x')\leq 2b_{k}+\frac{1}{4}\lambda_b d(x,x')\leq $$ $$\leq \frac{1}{4}\lambda_b(b_{k-1}-2b_k-b_{k+1})+\frac{1}{4}\lambda_b d(x,x')\stackrel{(*)}{\leq} \frac{1}{4}\lambda_b d(y,y')+\frac{1}{4}\lambda_bd(x,x')\leq $$ $$\leq  \frac{1}{2}\lambda_b d_m((x,y),(x',y')),
$$
where (*) follows from $d(y,y')\geq \on{dist}(Z_k,Z_j)\geq \on{dist}(X_k,X_j)\geq b_{k-1}-2b_k-b_{k+1}$.
\end{proof}
\begin{remark}\label{1814a}{
Observe that assumption (ii) of the above Lemma is satisfied if $h_k:Z\to X_k$ for all $2\leq k<\omega$. Indeed, in this case we have by Definition \ref{183}(iv), $d(h_k(x),h_j(x))\leq\on{diam}(\bigcup_{k\leq i<\omega}X_i)\leq b_{k-1}$. Then also $h_\omega(Z)=\{x_\omega\}$.}
\end{remark}
\begin{lemma}\label{1816}
Assume that $X,Y$ are $(\La,b,s)$ and $(\La',b',s)$-spaces such that $\lambda:=\sup\left\{\frac{b'_k}{b_k}:k\in\N\cup\{0\}\right\}<\infty$, and~let $M,M'$ be their compact subsets such that~$\La'_{M'}\subset \La_M$. Then there exists surjective map $h:M\to M'$ with $\on{Lip}(h)\leq 2\lambda$. \\
Moreover, if $X=Y$ (and $\La=\La'$ and $b=b'$) and $M'\subset M$, then additionally $h(x)=x$ for $x\in M'$.
\end{lemma}
\begin{proof}
At first, fix a map $T:\La'_{M'}\cup\tL'_{M'}\to\tL'_{M'}$ such that $\eta\preceq T(\eta)$ and if $\eta\notin \tL'_{M'}$, then the value $T(\eta)_{|\eta|+1}$ is biggest possible.\\
 Now for every $\eta\in\tL_M$, let $j:=\sup\{k:[\eta]_k\in\La'_{M'}\}$, and define $R(\eta):=T([\eta]_j)$, where we assume $[\eta]_\omega=\eta$. By assumptions and a fact that $T(\eta)=\eta$ for $\eta\in \tL'_{M'}$, it is easy to see that:
\begin{itemize}
\item[(a)] $R(\eta)=\eta$ for $\eta\in\tL'_{M'}\cap \tL_M$;
\item[(b)] $R(\tL_{M})=\tL'_{M'}$.
\end{itemize}
Finally, set $h:M\to M'$ by $h(x_\eta):=y_{R(\eta)}$, $\eta\in\tL_M$. By Observation \ref{1819}, $h$ is well defined, by (b), $h(M)=M'$, and by (a), if $X=Y$ and $M'\subset M$, we have that $h(x)=x$ for $x\in M'$. It remains to prove that $\on{Lip}(h)\leq 2\lambda$.\\
Choose any $\eta,\eta'\in\tL_{M}$ so that $R(\eta)\neq R(\eta')$. We can assume that there exist a~sequence $\xi$ and $k<k'\leq\omega$ such that $\xi\ha k\preceq \eta$ and $\xi\ha k'\preceq \eta'$. Observe that $\xi\ha k$ or $\xi\ha k'$ belongs to $\La'_{M'}$. Indeed, otherwise
$$
j:=\max\{i:[\eta]_i\in\La'_{M'}\}=\max\{i:[\eta']_i\in\La'_{M'}\}\leq|\xi|
$$ 
and then $R(\eta)=T([\eta]_j)=T([\eta']_j)=R(\eta')$, which contradicts our assumption.\\
Assume first that $\xi\ha k\in\La'_{M'}$. Then $R(\eta)=T([\eta]_j)\succeq [\eta]_j\succeq \xi\ha k$. Now if $\xi\ha k'\in \La'_{M'}$, then also $R(\eta')\succeq \xi\ha k'$.\\
If $\xi\ha k\notin \La'_{M'}$, then by definition of the map $T$ and a fact that $\xi\ha k\in\La'_{M'}$, there is $k\leq l\leq\omega$ such that $R(\eta')=T(\xi)\succeq \xi \ha l$. In both cases, we have
$h(x_\eta),h(x_{\eta'})\in \bigcup_{k\leq j\leq \omega}Y_{\xi\ha j}$ and hence
$$
d(h(x_\eta),h(x_{\eta'}))\leq\on{diam}{\Bigg(}\bigcup_{k\leq j\leq \omega}Y_{\xi\ha j}{\Bigg)}\leq b'_{(\xi)+k-1}.
$$
On the other hand, 
$$
d(x_\eta,x_{\eta'})\geq\on{dist}(X_{\xi\ha k},X_{\xi\ha k'})\geq b_{l(\xi)+k-1}-2b_{l(\xi)+k}-b_{l(\xi)+k+1}\geq \frac{1}{2}b_{l(\xi)+k-1}.
$$ 
Hence 
$$d(h(x_\eta),h(x_{\eta'})\leq 2\frac{1}{2}b_{l(\xi)+k-1}\frac{b'_{l(\xi)+k-1}}{b_{l(\xi)+k-1}}\leq 2\lambda d(x_\eta,x_{\eta'}).$$
In the case when $\xi\ha k'\in \La'_{M'}$, we proceed in similar way.
\end{proof}

\subsection{Scattered spaces as $(\La,b,s)$-spaces}\label{s31}
A topological space $X$ is called \emph{scattered} if every its nonempty subspace $Y$ has an isolated point in $Y$. It is well known that a compact metrizable space is scattered iff it is countable (it follows from the Cantor-Bendixson theorem and a fact that perfect sets are uncountable). 

For a scattered space $X$, define
$$X' := \{x\in X : x \textrm{ is an accumulation point of } X\}.$$
Then $X'$ is called \emph{the~Cantor-{Bendixson} derivative of $X$}. For each ordinal $\alpha$, we can define the Cantor-{Bendixson} $\alpha$-th derivative $X^{(\alpha)}$, according to the inductive procedure:
\begin{itemize}
\item[$\bullet$] $X^{(\alpha+1)} := \left(X^{(\alpha)}\right)'$;
\item[$\bullet$] $X^{(\alpha)} := \bigcap_{\beta < \alpha} X^{(\beta)}$ for a limit ordinal $\alpha$.
\end{itemize}
(the definition is correct, as sets $X^{(\alpha)}$ are empty from some level).
Then we can define \emph{the {scattered} height} of a scattered space $X$, by 
$$\on{ht}(X) := \min\{\alpha: X^{(\alpha)} \textrm{ is discrete}\}.$$
The classical Mazurkiewicz--Sierpi\'nski theorem (see \cite{MS}), states that every countable compact scattered space~$X$ is homeomorphic to the space $\omega^\beta \cdot n + 1$ with the order topology, where $\beta = \on{ht}(X)$ and~$n~=~\on{card}X^{(\beta)}$. In particular each countable compact scattered spaces with the same height and the same cardinality of elements with the ``highest rank'' are homeomorphic. We call a scattered space \emph{unital}, if $X^{(\on{ht}(X))}$ is a singleton.

Now we show that each compact metrizable scattered space is homeomorphic to certain $(\La,b,s)$-space.

Let $\delta_0$ be a countable limit ordinal, 
and $\{c_n^{\delta_0}(\beta): n\in\N, \beta\leq \delta_0,\;\beta\mbox{ is limit}\}$ be a monotone ladder system in $\delta_0$, that is, a family of ordinals that satisfies
\begin{itemize}
\item[1.] for each limit ordinal $\beta \leq \delta_0$, the sequence $\{c_n^{\delta_0}(\beta)\}_{n\in\N}$ is strictly increasing and converges to $\beta$;
\item[2.] for every limit $\beta, \gamma \leq \delta_0$, if $\beta \leq \gamma$ then $c_n^{\delta_0}(\beta) \leq c_n^{\delta_0}(\gamma)$ for every $n\in\N$.
\end{itemize}
(the proof of its existance is given in \cite{N}). Now for every ordinal $\alpha \leq \delta_0$, define the sequence $(\alpha_n)$ such that:
\begin{itemize}
\item[$\circ$] if $\alpha$ is a limit ordinal, then $\alpha_n:=c_n^{\delta_0}(\alpha)$;
\item[$\circ$] if $\alpha$ is a successor ordinal and $\alpha=\alpha'+1$, then $\alpha_n:=\min\{c_n^{\delta_0}(\alpha+\omega),\alpha'\}$.
\end{itemize}
Clearly,
\begin{itemize}
\item[(a)] for every $\alpha\leq\delta_0$, $(\alpha_n+1)$ is nondecreasing and converges to $\alpha$;
\item[(b)] for every $\alpha\leq\beta\leq\delta_0$ and every $n\in\N$, $\alpha_n\leq\beta_n<\beta$.
\end{itemize}

Now we will define a family $\Lambda^\alpha$, $\alpha\leq\delta_0$, of proper trees.
Definition will be recursive. At first, let $\Lambda^0:=\{\emptyset\}$. Assume that for some $\alpha\leq\delta_0$, all sets $\Lambda^\beta$, $\beta<\alpha$, are defined. Then define
$$\Lambda^\alpha:=\{\emptyset,\omega\}\cup\bigcup_{k\in\N}k\ha\Lambda^{\alpha_k}=\bigcup_{k\in\N}\{k\hat{\;}\eta:\eta\in\Lambda^{\alpha_k}\} \cup \{\emptyset,\omega\},$$
where $k\ha A:=\{k\ha a:a\in A\}$.

The following proposition lists basic properties of sets $\Lambda^\alpha$. 
 It can be easily proved by  induction and with a help of (a) and (b).
\begin{proposition}\label{1}
If $\alpha\leq\beta\leq\delta_0$, then
\begin{itemize}
\item[(i)] $\tL^\beta=\{\omega\}\cup\bigcup_{k\in\N}k\ha\tL^{\beta_k}$;
\item[(ii)] $\tilde{\Lambda}^\beta$ consists of finite sequences, that is, $\tL^\beta\subset\La^\beta$;
\item[(iii)] $\La^\beta$ is proper;
\item[(iv)] $\La^\alpha\subset \La^\beta$.
\end{itemize}
\end{proposition}
Finally, set $\La^{\alpha,1}:=\La^\alpha$ and for $n\geq 2$, define $\La^{\alpha,n}$ by
$$
\La^{\alpha,n}:={\Bigg(}\bigcup_{1\leq i\leq n-1}i\ha\La^\alpha{\Bigg)}\cup{\Bigg(}\bigcup_{i\geq n}i\ha\La^{\alpha_{i-n+1}}{\Bigg)}\cup\{\emptyset,\omega\}
$$ 
Clearly, $\Lambda^{\alpha,n}$ is proper tree. 
As $\delta_0$ was taken as arbitrary countable ordinal, the following result shows that all compact scattered spaces are certain $(\La,b,s)$-spaces.
\begin{proposition}\label{189}
If $b$ is a good sequence and $\alpha\leq\delta_0$, then $(\La^{\alpha,n},b,s)$-space is homeomorphic to $\omega^\alpha\cdot n+1$.
\end{proposition}
\begin{proof}
We first assume $n=1$. Then the thesis holds by a simple inductive argument and the fact that if $X$ is a $(\La^\alpha,b,s)$-space, then for every $k\in\N$, $X_k$ is a $(\La^{\alpha_k},b^k,s)$-space (see Lemma \ref{1822}{(d)}). Hence take $(\La^{\alpha,n},b,s)$-space $X$. 
Then from the definition it can be easily seen (see again Lemma \ref{1822}{(d)}) that for $i=1,{\dots},n-1$, the space $X_i$ is a $(\La^\alpha,b^i, s)$-space and that the space $\tilde{X}_n:=\bigcup_{n\leq i\leq\omega}X_i$ is a~$(\La^\alpha,b^{n-1},s)$-space. Since $X_1,{\dots},X_{n-1},\tilde{X}_n$ are pairwise disjoint and open, the result follows.
\end{proof}

\begin{remark}
{Nowak in \cite{N} considered a particular case of $(\La^\alpha,b,s)$-spaces. Namely, take $r>3$ and for each $n\in\N$, define the affine homeomorphism $s_n$ by
$$s_n(x):=\frac{x}{r^n} + \frac{1}{r^n}.$$
Finally, construct the scattered compact sets $L_\alpha \subset [0,1]$, $\alpha\leq \delta_0$, in the following recursive way:
\begin{equation}\label{L}
\left.\begin{array}{ll}
\textrm{1. } L_0 := \{0\}; \\
\textrm{2. } L_\alpha := L_0 \cup \bigcup_{n=1}^\infty s_n(L_{\alpha_n}).
 \end{array}\right.
\end{equation}
Notice that $L_\alpha$ is a $(\La^\alpha,b,s)$-space for a sequence $b=\left(\frac{2}{r^{k+1}}\right)_{k\in \N\cup\{0\}}$ with sufficiently large $r$. From the~perspective of countable compact spaces and Section 4, Nowak's construction would be sufficient. 
However, in next sections we will need our, more flexible, setting.}
\end{remark}

\subsection{Compact $0$-dimensional uncountable spaces as $(\La,b,s)$-spaces} 

In this section we show that each compact metrizable uncountable $0$-dimensional space is homeomorphic to a certain subset of a $(\La_{\max},b,s)$-space. Note that if $Y$ is uncountable, compact, metrizable and $0$-dimensional, by the Cantor-{Bendixson} theorem, there exists a maximal perfect subset $\iY$ (called a perfect kernel) which is a Cantor space.\\
Define two subtrees of $\La_{\max}$
\begin{equation}\begin{array}{c}\label{1813}\La_{s}:=\{\eta\in\La_{\max}:\eta\mbox{ does not contain }1\}\\\mbox{ and }\\\La_r:=\{\eta\in\La_{\max}:\eta\mbox{ contains at most one }1\}\end{array}
\end{equation}
Observe that
\begin{equation}\begin{array}{c}\label{1813'}\tL_{s}:=\{\eta\in\tL_{\max}:\eta\mbox{ does not contain }1\}
\\\mbox{ and }\\\tL_r:=\{\eta\in\tL_{\max}:\eta\mbox{ contains at most one }1\}\end{array}
\end{equation}
\begin{proposition}\label{1813}
Let $b$ be a good sequence, $Y$ be an uncountable $0$-dimensional compact metrizable space and $y_0\in Y^{(\infty)}$. Then for every $(\La_{\max},b,s)$-space $X$, there is a compact set $M\subset X$ such that 
\begin{equation}\label{1812}
   \La_s\subset\La_M\subset\La_r
\end{equation}
and a homeomorphism $h:Y\to M$ such that $h(y_0)=x_\omega$. 
\end{proposition}
Before we give the proof, let us observe that (\ref{1812}) implies
\begin{equation}\label{1812'}
   \tL_s\subset\tL_M\subset\tL_r
\end{equation}
\begin{proof}
Let $X$ be a $(\La_{\max},b,s)$-space and define
$$
C_s:=\{x_\eta:\eta\in\tL_s\}\;\;\;\mbox{and}\;\;\;C_r:=\{x_\eta:\eta\in\tL_r\}.
$$ Clearly, $\La_{C_s}=\La_s$ and $\La_{C_r}=\La_r$. Also, $C_s,C_r$ are closed in $X$, as 
$$\begin{array}{c}C_s=X\setminus \bigcup\{X_\eta:\eta\in\La_{\max},\;\textrm{$\eta$ ends with $1$}\} \\\textrm{  and  }\\ C_r=X\setminus \bigcup\{X_\eta:\eta\in\La_{\max}, \;\textrm{$\eta$ ends with $1$ and contains two $1$'s}\}\end{array}$$ and the sets which we remove are open. Moreover, it is easy to see that $C_s,C_r$ are also perfect, hence they are Cantor spaces. 
Finally, $C_s\subset C_r$ and $C_s$ has empty interior with respect to $C_r$. Indeed, take $x=x_\eta\in C_s$. If $\eta=\tilde{\eta}\ha\omega$, then for every $k\geq 2$, $x_k:=x_{\tilde{\eta}\ha k\ha 1\ha(2,2,{\dots})}\in C_r\setminus C_s$ and $x_k\to x$, and if $\eta$ is infinite, then for every $k\in\N$, $x_k:=x_{[\eta]_k\ha 1\ha(2,2,{\dots})}\in C_r\setminus C_s$ and $x_k\to x$.\\
Hence the result follows from the following Claim (whose version we also prove in \cite{BNS2}; we give a proof for the sake of completeness):\\
\textbf{Claim}. Assume that $C_1,C_2$ are Cantor spaces such that $C_1\subset C_2$ and $\on{Int}(C_1)=\emptyset$. Then there exists a homeomorphic embedding $h:Y\to C_2$ such that $h(\iY)=C_1$ and $h(y_0)=x_0$, where $y_0\in \iY$ and $x_0\in C_1$ are initially chosen.
\newline
Since $\iY$ and $C_1$ are Cantor spaces, there is a homeomorphism $h:\iY\to C_1$ such that $h(y_0)=x_0$.
Consider a family $$D:=\{f:Y\to C_2:f\;\mbox{is continuous and}\;f_{\vert\iY}=h\}$$ and consider it as a metric space with the supremum metric. Let us note that $D$ is nonempty as the set $\iY$ is a retract of zero dimensional space $Y$ (see \cite[Theorem 7.3]{Ke}), so there exists a continuous retraction $r:Y\to\iY$, and $h\circ r\in D$. Now let $z_1,z_2,{\dots}$ be all elements of (clearly countable) set $Y\setminus\iY$ (in the case when $Y\setminus \iY$ is finite, we can proceed in similar, but simpler way), and for every $n,m\in\N$, let
$$
U_n:=\{f\in D:f(z_n)\notin C_1\},\;\;U_{n,m}:=\{f\in D:f(z_n)\neq f(z_m)\}.
$$
Observe that any map $f\in\bigcap_{n,m\in\N}(U_n\cap U_{n,m})$ satisfies the thesis of the Claim. Hence, in view of~completeness of $D$, it is enough to prove that all $U_n$ and $U_{n,m}$ are open and dense. Openness of these sets is obvious. We will show that $U_n$ is dense. Take $f\in D$ and $\varepsilon>0$, and choose $y_0\in C_2\setminus C_1$ so that $d(f(z_n),y_0)<\varepsilon$ (where $d$ is an original metric on $C_2$). Then take a clopen set $V\subset Y \setminus \iY$ such that $z_n\in V$ and $f(V)\subset B(f(z_n),\frac{\varepsilon}{2})$, and define
$$
\tilde{f}(x):=\left\{\begin{array}{ccc}f(x)&\mbox{if}&x\notin V,\\
y_0&\mbox{if}&x\in V.\end{array}\right.
$$
As $V$ is clopen and $V\cap \iY=\emptyset$, $\tilde{f}$ is continuous and therefore $\tilde{f}\in D$. Moreover, for every $x\in Y$,\\
- if $x\notin V$, then $d(f(x),\tilde{f}(x))=0$;\\
- if $x\in V$, then $d(f(x),\tilde{f}(x))=d(f(x),y_0)\leq d(f(x),f(z_n))+d(f(z_n),y_0)<\varepsilon$.
Hence the supremum metric $d_s(f,\tilde{f})\leq\varepsilon$ and $U_n$ is dense. Similarly we prove the density of $U_{n,m}$. Claim is proved.
\end{proof}

\section{$0$-dimensional compact metrizable spaces as attractors of Banach GIFSs of order $2$}
The main result of this section is the following:
\begin{theorem}\label{scattered}
Let $X$ be a compact metrizable $0$-dimensional space and $\lambda\in(0,1)$. Then $X$ is homeomorphic to the attractor of a GIFS $\F$ of order $2$ on the real line with $\on{Lip}(\F)\leq\lambda$.\\
Moreover, if $X$ is scattered, then $\F$ can consist of two maps.
\end{theorem}

The aforementioned theorem will follow automatically from Theorem \ref{scattereddok1}, which give more "qualitative" result.
\begin{remark}{
(1) As was mentioned in Theorem \ref{nowak1}, Nowak in \cite{N} showed that any compact scattered metrizable space of limit {scattered} height is not a~topological fractal. Hence Theorem \ref{scattered} shows that scattered spaces of limit { scattered height distinguish} the classes of~GIFSs' fractals and IFSs' fractals.\\
(2) As was mentioned in Theorems \ref{nowak1} and \ref{bns1}, $0$-dimensional compact metrizable spaces which are uncountable or scattered with successor height, are homeomorphic to attractors of IFSs on the real line consisting Banach contractions. In the constructions presented there, the number of required mappings in case of uncountable or unital case is two (but in a nonunital case it seems that we need at least three maps - see \cite{BNS2}). Hence our result, in the case of uncountable spaces, says nothing new. However, later we will need some particular properties that are guaranteed by our construction, so the proof of Theorem \ref{scattered} for an uncountable case can be considered as a step in the proof of results presented in the next section.\\
(3) In \cite{DISS}, Strobin and his coauthors studied a topological version of a GIFS, and Theorem \ref{scattered} shows that all compact scattered spaces are attractors of such kind of GIFSs.}
\end{remark}
The promised qualitative version of Theorem \ref{scattered} is the following (then Theorem \ref{scattered} follows from it and Propositions \ref{189}, \ref{1813}, Lemma \ref{nl2} and Theorem \ref{fvf3}).
\begin{theorem}\label{scattereddok1}$\;$\\
(1) Let $\alpha\leq\delta_0$, $n\in\N$ and $b$ be a good sequence. Then each $(\La^{\alpha,n},b,s)$-space is the attractor of~a~Banach GIFS $\F$ of order $2$ which consists of two maps and such that $\on{Lip}(\F)\leq \lambda_b$.\\
(2) Let $X$ be a $(\La_{\max},b,s)$-space and $M\subset X$ be compact set which satisfies (\ref{1812}). Then $M$ is the attractor of a Banach GIFS $\F$ of order $2$ which consists of four maps and such that $\on{Lip}(\F)\leq \lambda_b$.
\end{theorem}

In the {remaining part} of this section we prove Theorem \ref{scattereddok1}. 
\begin{proof}
Assume now that $X$ is a $(\La^\alpha,b,s)$-space and choose $k\in\N$. By Proposition \ref{1} we see that $\Lambda^{\alpha_k}\subset \Lambda^\alpha$ and, as was already observed, the space $X_k$ is $(\La^{\alpha_k},b^k,s)$-space. Moreover, $$\lambda:=\sup\left\{\frac{b^k_n}{b_n}:n\in\N\right\}=\sup\left\{\frac{b_{n+k}}{b_n}:n\in\N\right\}\leq M_b\leq\frac{1}{8}\lambda_b.$$  
Hence by Lemma \ref{1816} there exists a surjective map $h_k:X\to X_k$ with $\on{Lip}(h_k)\leq\frac{1}{4}\lambda_b$. Then by Lemma \ref{1814} and Remark \ref{1814a}, there exists a map $F:X\times X\to X$ such that $\on{Lip}(F)\leq \frac{1}{2}\lambda_b$ and $F(X\times X)=X\setminus X_1$. Moreover, define $G:X\to X$ by $G(x):=h_1(x)$. Then $G(X)=X_1$ and $\on{Lip}(G)\leq\frac{1}{4}\lambda_b$. Hence by~Lemma~\ref{nl1}, the proof of Theorem \ref{scattereddok1}(1) in the case when $n=1$ is finished.\\
Now let $X$ be a $(\La^{\alpha,n},b,s)$-space, where $n>1$. As was observed, $X_1,X_2,{\dots},X_{n-1}$ are $(\La^\alpha,b^i,s)$-spaces, for $i=1,2,{\dots},n-1$, respectively, and $\tilde{X}_n:=\bigcup_{n\leq k\leq \omega}X_k$ is a $(\La^{\alpha},b^{n-1},s)$-space.
\newline
By what we already proved, there is a surjection $F_1:X_1\times X_1\to X_1\setminus X_{(1,1)}$ with $\on{Lip}(F_1)\leq \frac{1}{2}\lambda_{b^1}\leq \frac{1}{2}\lambda_{b}$. Now (see Lemma \ref{1822}{(d)}) since $X\setminus X_1$ is a $(\La^{\alpha,n-1},b^1,s)$-space, $X_{(1,1)}$ is $(\La^{\alpha_1},b^2,s)$-space and $\La^{\alpha_1}\subset \La^{\alpha,n-1}$, by Lemma \ref{1816} there exists a surjection $h:X\setminus X_1\to X_{(1,1)}$ with $\on{Lip}(h)\leq 2M_b\leq \lambda_b$. Hence, by simple calculations (involving (\ref{181b})) we can show that the map $P:X\times X\to X_1$ defined by
$$
P(x,y):=\left\{\begin{array}{ccc}F_1(x,y)&\mbox{if}&(x,y)\in X_1\times X_1,\\
h(x)&\mbox{if}&x\notin X_1,\\
z&\mbox{if}&x\in X_1, \;y\notin X_1\end{array}\right.
$$
where $z$ is an arbitrary point of $X_1$, is surjective and $\on{Lip}(P)\leq \lambda_b$.\\
Now if $n=2$, then exactly the same reasoning, shows that there exists a surjection $Q:X\times X\to \tilde{X}_2=X\setminus X_1$ with $\on{Lip}(Q)\leq \lambda_b$, and we are done. Hence let $n\geq 3$.\\
By Lemma \ref{1816}, proceeding similarly as earlier, we can see that for $i=1,{\dots},n-2$, there exists a surjective map $h_i:X_i\to X_{i+1}$ with $\on{Lip}(h_i)\leq 2M_b\leq \lambda_b$. Moreover, for the same reason, there exists a surjection $h_{n-1}:X_1\to \tilde{X}_n$ with $\on{Lip}(h_{n-1})\leq2M_b\leq \lambda_b$. Then the map $Q:X\times X\to X\setminus X_1$ defined by
$$
Q(x,y):=\left\{\begin{array}{ccl}h_i(x)&\mbox{if}&x\in X_i,\;i=1,{\dots},n-2,\;y\in X_1,\\
h_{n-1}(x)&\mbox{if}&x\in X_1,\;y\notin X_1,\\
z&\mbox{otherwise}\end{array}\right.
$$
where $z$ is an arbitrary element of $X_{n-1}$,
is surjective and $\on{Lip}(Q)\leq\lambda_b$.\\
All in all, the proof of Theorem \ref{scattereddok1}(1) is finished.

Now we switch to the case when $M$ is a compact subset of a $(\La_{\max},b,s)$-space $X$ which satisfies (\ref{1812}).\\
Let $Z:=\{x_\eta:\eta\in\tL_s\}$. By (\ref{1812'}) and Observation \ref{1819}, $Z\subset M$. We will define four maps: $G_i:Z\times Z\to M$, $i=2,3,4$ and $F:Z\to M$, such that $\on{Lip}(G_i),\on{Lip}(F)\leq \frac{1}{2}\lambda_b$ and $\bigcup_{i=2}^4G_i(Z\times Z)\cup F(Z)=M$. This will end the proof. Indeed, then we can define $\tilde{G}_i:M\times M\to M$ and $\tilde{F}:M\to M$ by $\tilde{G}_i(x,y)=G_i(r(x),r(y))$ and $\tilde{F}(x)=F(r(x))$, where $r:X\to Z$ is a projective map with $r(x)=x$ for $x\in Z$ and $\on{Lip}(r)\leq 2$ (existence of $r$ is guaranteed by Lemma \ref{1816}). Then we obtain GIFS $\F$ with $\on{Lip}(\F)\leq \lambda_b$ whose attractor is $M$.\\
Define $M_1,M_2,M_3,M_4$ in the following way:
$$
M_4:=M\setminus (X_1\cup X_2\cup X_3)
$$
and 
$$
M_i:=M\cap X_i,\;i=1,2,3.
$$
Clearly, $M=\bigcup_{i=1}^4 M_i$. We first define the map $G_4$. 
If $\eta\in\tL_s$, then by $\eta'$ we will mean the sequence so that $\eta=\eta_1\ha \eta'$. Moreover, if $2\leq k<\omega$ and $|\eta|\geq k$, then by $\eta^1,\eta^2$ we will mean sequences so that $\eta=\eta_1\ha \eta^1\ha\eta_k\ha\eta^2$ (if $k=2$ and $|\eta|=2$, then $\eta^1=\eta^2=\emptyset$). 
\newline
Let $s_k:Z\to X$ be the map defined by (we assume $\omega+2=\omega$ and $\omega-1=\omega$ and equate $x_\eta\in Z$ with $\eta \in \tL_{s}$):
$$
s_k(x_\eta):=\left\{\begin{array}{ccc}x_{(\eta_1+2)\ha\eta'}&\mbox{if}&|\eta|< k,\\
x_{(\eta_1+2)\ha\eta^1\ha(\eta_k-1)\ha \eta^2}&\mbox{if}&\eta=\eta_1\ha\eta^1\ha\eta_k\ha\eta^2.
\end{array}\right.$$
Observe that $s_k$ is well defined as $X$ is $(\La_{\max}, b, s)$-space and $\eta$ does not contain $1$. Moreover, $\on{Lip}(s_k)\leq\frac{1}{8}\lambda_b$. Indeed, take distinct $x=x_\eta,\;y=y_\xi\in Z$, and let $j$ be such that $[\eta]_j=[\xi]_j$ and $x\in  X_{[\eta]_j\ha p}$ and $y\in X_{[\xi]_j\ha q}$ for some $p<q\leq\omega$. We consider two Cases:\\
Case 1. $j> k-1$. Then, setting $\beta:=[\eta]_j$ and $\tilde{\beta}$ -- its modification due to definition of $s_k$ (that is, $\tilde{\beta}=(\eta_1+2)\ha[\eta]_j^1\ha(\eta_k-1)\ha[\eta]^2_j$) -- we have by~(\ref{181b})
$$
d(x,y)\geq \on{dist}\left(X_{\beta\ha p},X_{\beta\ha q}\right)\geq b_{l(\beta)+p-1}-2b_{l(\beta)+p}-b_{l(\beta)+p+1}
$$
and
$$
d(s_k(x),s_k(y))\leq \on{diam}{\Bigg(}\bigcup_{p\leq i\leq\omega}X_{\tilde{\beta}\ha i}{\Bigg)}\leq b_{l(\tilde{\beta})+p-1}=$$ $$=b_{l(\beta)+2-1+p-1}=b_{l(\beta)+p}\leq\frac{1}{8}\lambda_b(b_{l(\beta)+p-1}-2b_{l(\beta)+p}-b_{l(\beta)+p+1})\leq\frac{1}{8}\lambda_b d(x,y).
$$
Case 2. $j=k-1$. Then, letting $\tilde{\beta}$ be appropriate modification of $\beta$, we have that
$$
d(s_k(x),s_k(y))\leq \on{diam}{\Bigg(}\bigcup_{p-1\leq i\leq\omega}X_{\tilde{\beta}\ha i}{\Bigg)}\leq b_{l(\tilde{\beta})+p-2}=$$ $$=b_{l(\beta)+2+p-2}=b_{l(\beta)+p}\leq\frac{1}{8}\lambda_b(b_{l(\beta)+p-1}-2b_{l(\beta)+p}-b_{l(\beta)+p+1})\leq\frac{1}{8}\lambda_b d(x,y).
$$
Case 3. $j<k-1$. Then we obtain the desired inequality in a similar way.

Finally, we show that for $k<j<\omega$ and every $x\in Z$, $d(s_k(x),s_j(x))\leq b_{k+1}$. Let $x=x_\eta$. If $|\eta| < k$, then also $|\eta| < j$, and $s_k(x)=s_j(x)$. Hence assume that $|\eta|\geq k$. Then $s_k(x)\in X_{(\eta_1+2,\eta_2,{\dots},\eta_{k-1})}$ and~$s_j(x)\in X_{(\eta_1+2,\eta_2,{\dots},\eta_{k-1})}$, so
$$
d(s_k(x),s_j(x))\leq\on{diam}\left(X_{(\eta_1+2,\eta_2,{\dots},\eta_{k-1})}\right)\leq b_{\eta_1+2+\eta_2+{\dots}+\eta_{k-1}}\leq b_{k-1+2}=b_{k+1}.
$$
Finally, let $r':X\to M$ be the projective map with $r'(x)=x$ for $x\in M$ and $\on{Lip}(r')\leq 2$, and set $s_k':=r'\circ s_k$. Then $\on{Lip}(s'_k)\leq \frac{1}{4}\lambda_b$ and $d(s_{k}'(x),s_{j}'(x))\leq 2b_{k+1}\leq b_{k}$. 
Hence the assumptions of~Lemma~\ref{1814} are satisfied for $h_k:=s'_{k-1},\;k\geq 3$ and $h_2$ - any constant map with value in $M_4$. Hence there is a map $G_4:Z\times Z\to X$ such that $\on{Lip}(G_4)\leq\frac{1}{2}\lambda_b$, and (since $Z\cap X_i\neq\emptyset$ for $i\geq 2$)
$$M\supset G_4(Z\times Z)=\bigcup\{h_k(Z):k\in\oN,\;k\geq 3\}\supset\bigcup\{s_k'(Z):k\geq 2\}\supset M_4$$
To see the last inclusion, take any $x_\eta\in M_4$. If $|\eta|\geq 2$, then we can choose $k\geq 2$ so that $\eta_k=\min\{\eta_i:i\geq 2\}$ and we put $\xi:=(\eta_1-2)\ha\eta^1\ha (\eta_k+1)\ha\eta^2$ and then $s_k'(x_\xi)=s_k(x_\xi)=x_\eta$ (since $\eta$ can have just one $1$, we have that $x_\xi\in Z$). If $|\eta|=1$, then $\eta=\omega$ and $s_2'(x_\omega)=x_\omega$.\\
Now for $i=2,3$ and every $k\geq 2$, define $s^i_k:Z\to X$ by
$$
s^i_k(x_\eta):=\left\{\begin{array}{ccc}x_{i\ha\eta}&\mbox{if}&|\eta|< k-1,\\
x_{i\ha\eta^1\ha(\eta_{k-1}-1)\ha \eta^2}&\mbox{if}&\eta=\eta^1\ha\eta_{k-1}\ha\eta^2.
\end{array}\right.
$$
Similarly as before we can show that $\on{Lip}(s^i_k)\leq \frac{1}{8}\lambda_b$ and $d(s_k^i(x),s_j^i(x))\leq b_{k}$, and proceeding in an~analogous way, we can define maps $G_i:Z\times Z\to X$ with $\on{Lip}(G_i)\leq \frac{1}{2}\lambda_b$ and such that for $i=2,3$, $M_i~\subset G_i(Z\times Z)\subset~M$.\\
Finally, define $F:Z\to M$ by $F(x_\eta):=r'(x_{1\ha\eta})$. Then it is easy to see that $\on{Lip}(F)\leq \frac{1}{2}\lambda_b$ and $F(Z)\supset M_1$. The result follows.
\end{proof}

\begin{remark}{By Lemma \ref{filipnowy9} we can see that appropriately separated union of $n$ many $(\La^\alpha,b,s)$-sets is the attractor of some Banach GIFS. However, presented construction shows that it is enough to take just two functions.}
\end{remark}

\begin{remark}{
It is worth underlining that GIFSs give much more possibilities than IFSs. Considering Cartesian {power} in the domain provides indispensable space when considering compact scattered spaces with countable limit height. Moreover, observe that the definition of function $G$ (in the proof for $(\La^\alpha,b,s)$-spaces) does not make use of~whole Cartesian {power} (it can be perceived as a~selfmap).} 
\end{remark}

\section{$0$-dimensional compact spaces as attractors of GIFSs of order $m$ and nonattractors of GIFSs of order $m-1$}

In this section we prove the following theorem:
\begin{theorem}\label{m-scattered}
Let $X$ be an infinite compact $0$-dimensional metrizable space.\\
(1) For every $\lambda\in(0,1)$ and $m\geq 2$, there exists $Z\subset \R$ which is homeomorphic to $X$ and such that:
\begin{itemize}
\item[(1a)] $Z$ is the attractor of some GIFS $\F$ on the real line of order $m$ with $\on{Lip}(\F)\leq\lambda$;
\item[(1b)] $Z$ is not {the} attractor of any weak GIFS of order $m-1$.
\end{itemize}
(2) There exists $Z\subset \R$ which is homeomorphic to $X$ and which is not {the} attractor of any weak GIFS.
\end{theorem}
Almost whole Theorem \ref{m-scattered} will follow from more ``qualitative$"$ Theorem \ref{m-scattereddok} and Lemma \ref{nl2} (or \cite[Lemma 1.1(ii)]{L}). 
However, the case when $X\setminus X^{(\infty)}$ is finite (where $X^{(\infty)}$ is the perfect kernel) need to be dealt separately. 
\newline
At first recall that in \cite{S}, for every $\lambda\in(0,1)$ and every $m\geq 2$, we constructed a Cantor subset of the plane which is the attractor of some Banach GIFS of order $m$ which is not the attractor of any weak GIFS of order $m-1$, and also we constructed a Cantor set which is not the attractor of any weak GIFS. As can be seen, the construction proved there can be modified so that we get appropriate Cantor subsets of the real line (in fact, in the first draft of that paper, exactly such sets were constructed, as they were inspired by the construction from \cite{CR}). Hence, Theorem \ref{m-scattered}(1) for $X$ with $X\setminus X^{(\infty)}$ finite, follows from Lemma \ref{filipnowy9}, Remark \ref{nr1} and the following:
\begin{lemma}
Assume that $m\in\N$, $\emptyset\neq C\subset \R$, $\emptyset\neq P\subset \R$ is finite and $\on{dist}(C,P)\geq \on{diam}(C)$. If $C\cup P$ is the attractor of a weak GIFS of order $m$, then $C$ is the attractor of a weak GIFS of order $m$.
\end{lemma}
\begin{proof}
By our assumptions, we can define the projection $\pi:C\cup P\to C$ with $\on{Lip}(\pi)=1$.
Let~$\F~=~\{f_1,{\dots},f_n\}$ be a weak GIFS of order $m$ such that $C\cup P$ is its attractor. Now for every $1\leq k<m$, $1\leq i_1<{\dots}<i_k\leq m$, $(z_{i_1},{\dots},z_{i_k})\in P^k$ and $j=1,{\dots},n$, we define $f_j^{(i_1,{\dots},i_k),(z_{i_1},{\dots},z_{i_k})}:C^{m-k}\to C$ in the following way: if $(x_1,{\dots},x_{m-k})\in C^{m-k}$, then let $f_j^{(i_1,{\dots},i_k),(z_{i_1},{\dots},z_{i_k})}(x_1,{\dots},x_{m-k})$ be equal to the~~value $\pi(f_j(y_1,{\dots},y_m))$, where $y_{i_1}=z_{i_1},{\dots},y_{i_k}=z_{i_k}$ and remaining coordinates of $(y_1,{\dots},y_n)$ equals $x_1,{\dots},x_{m-k}$, consecutively. Define also $f_j^0: C^m \to C$ as $f_j^0 := \pi \circ f_j$. As $f_j$ are generalized weak contractions, so~are~$f_j^0$ and $f_j^{(i_1,{\dots},i_k),(z_{i_1},{\dots},z_{i_k})}$.  Finally, observe that
$$
C\setminus \bigcup_{j=1,{\dots},n}f_j(P^m)\subset $$ $$\subset \bigcup_{j=1,{\dots},n}{\bigg(} f_j^0\left(C^m\right) \cup \bigcup_{1\leq k<m}\bigcup_{i_1<{\dots}<i_k} \bigcup_{(z_{i_1},{\dots},z_{i_k})}f_j^{(i_1,{\dots},i_k),(z_{i_1},{\dots},z_{i_k})}\left(C^{m-k}\right) {\bigg)}\subset C
$$
Since $\bigcup_{j=1,{\dots},n}f_j(P^m)$ is finite, we can define finitely many constant maps $g_1,{\dots},g_l:C^m\to C$ so that these mappings, together with maps $f_j^0$ and obvious extensions (like in Lemma \ref{nl1}) of all maps $f_j^{(i_1,{\dots},i_k),(z_{i_1},{\dots},z_{i_k})}$, will form a~weak GIFS of order $m$ whose attractor is $C$, and this would be a contradiction.
\end{proof}

To prove Theorem \ref{m-scattered} in the {remaining case} when $X\setminus X^{(\infty)}$ is infinite, we have to extend a bit the notion of $(\La,b,s)$-space by adding finite set to each segment $X_k$.
We say that a pair of sequences $(b,p)$ is \emph{a good pair}, if $b$ is good and $p$ is~a~sequence of natural numbers such that $p_k\geq 2$ for $k\in\N$ and
\begin{equation}\label{good}
\sup\left\{8(1+4p_{k}){\frac{b_{k}}{b_{k-1}}}:k\in\N\right\}  < 1\;\;\;\mbox{and}\;\;\;\sup\left\{\frac{b_{k}}{b_{k-1}}p_k:k\in\N,\;k\geq 2\right\}\leq\lambda_b.
\end{equation}

\begin{definition}\label{1823}\emph{Let $(b,p)$ be a good pair and $\La$ be a proper tree. We say that a metric space $Z$ is a }$(\La,(b,p),s)$-space,\emph{ if $Z=X\cup Y$, where
\begin{itemize}
\item[(i)] $X$ is a $(\La,b,s)$-space;
\item[(ii)] $Y=\bigcup_{1\leq k\leq\omega}Y_k$, where $Y_\omega=X_\omega$, and for every $k\in\N$,
\begin{itemize}
\item[(iia)] $Y_k=\{a^k_1,{\dots},a^k_{p_k}\}$ for some $a^k_1,{\dots},a^k_{p_k}$ with $d(a^k_i,a^k_{j})=2(j-i)b_k$ for $i\leq j$;
\item[(iib)] $\on{dist}(X_k,Y_k)=2b_kp_k$.
\end{itemize}
\end{itemize}
}
\end{definition}
\begin{remark}\label{nr2}{
It is easy to see that for every sequence $p=(p_k)$ of natural numbers with $p_k\geq 2$ for $k\in\N$, and $\lambda>0$, there exists a good sequence $b$~such that the pair $(b,p)$ is good and $\lambda_b<\lambda$ (the important fact is that in the "middle condition" in (\ref{good}) we consider only $k\geq 2$). It is also clear that for every good pair $(b,p)$ and a proper tree $\Lambda$, there exists a set $X\subset \R$ which is a~$(\La,(p,b),s)$-space.}
\end{remark}
Later, when writing a $(\La,(b,p),s)$-space as $X\cup Y$, we will automatically assume that $X,Y$ have meaning as in the above definition. We start with making basic observations of the structure of $(\La,(b,p),s)$-spaces. If $Z=X\cup Y$ is a~$(\La,(b,p),s)$-space, then for any $k\in\oN$, we put $Z_k:=X_k\cup Y_k$.
\begin{observation}\label{obs2}
Let $Z=X\cup Y$ be a $(\La,(b,p),s)$-space with $p_n\geq 2$. Then for every $1\leq k<\omega$,
\begin{itemize}
\item[(i)] $\on{diam}(Z_k)\leq \frac{1}{8}b_{k-1}$;
\item[(ii)] 
$\on{dist}{\big(}Z_k,\bigcup_{k< j\leq\omega}Z_j{\big)}\geq  \frac{2}{3}b_{k-1}$;
\item[(iii)] $\on{diam}{\big(}\bigcup_{k\leq j\leq\omega}Z_j{\big)}\leq \frac{5}{4}b_{k-1}$;
\item[(iv)] $\on{diam}(X_k)\leq\on{diam}(Y_k)\leq\on{dist}(X_k,Y_k)$.
\end{itemize}
\end{observation}
\begin{proof}The properties follow from (\ref{good}), Definition \ref{183} and the fact that $b$~is good. (i) follows from the fact that for any $x,y\in Z_k$,
$$
d(x,y)\leq\on{diam}(X_k)+\on{dist}(X_k,Y_k)+\on{diam}(Y_k)\leq $$ $$\leq b_k+2b_kp_k+2b_k(p_k-1)\leq b_k(4p_k+1)\leq \frac{1}{8}b_{k-1}.
$$
To see (ii), note that for every $k<j\leq\omega$,
$$
\on{dist}(Z_k, Z_j)\geq \on{dist}(X_j,X_k)-\on{diam}(Z_j)-\on{diam}(Z_k)\geq $$ $$\geq  b_{k-1}-2b_k-b_{k+1}-\frac{1}{8}b_{j-1}-\frac{1}{8}b_{k-1}\geq \frac{7}{8}b_{k-1}-3b_k\geq \frac{2}{3}b_{k-1}.
$$
To see (iii), we calculate as follows:
$$\on{diam}{\Bigg(}\bigcup_{k\leq j\leq \omega}Z_k{\Bigg)}\leq b_{k-1}+2 \sup\{4b_jp_j:j\geq k\}\leq \frac{5}{4}b_{k-1}.$$
Point (iv) follows from definition.

\end{proof}

\begin{proposition}Let $(b,p)$ be a good pair, and $K$ be a compact, metrizable, $0$-dimensional space such~that $K\setminus K^{(\infty)}$ is infinite. There exists a proper tree $\Lambda$ such that for every $(\La,(b,p),s)$-space $Z=X\cup Y$, there exists a set $M\subset X$ such that
\begin{itemize}
\item[(i)] $M$ is the attractor of some Banach GIFS $\F$ of order $2$ with $\on{Lip}(\F)\leq \lambda_b$;
\item[(ii)] $M\cup Y$ is homeomorphic to $K$;
\item[(iii)] $M\cap X_k\neq\emptyset$ for $2\leq k\leq \omega$.
\end{itemize}
\end{proposition} 
\begin{proof} We will consider three cases:\\
Case 1. $K^{(\infty)}=\emptyset$.\\
This means that $Z$ is countable and infinite. Then it is scattered so by the Mazurkiewicz--Sierpi\'nski theorem and Proposition~\ref{189}, there is a proper tree $\La:=\La^{\alpha,n}$ such that $K$ is homeomorphic to any $(\La,b,s)$-space. Hence let $X\cup Y$ be a $(\La,(b,p),s)$-space. By Theorem~\ref{scattereddok1}(1), $X$ is the~attractor of some Banach GIFS $\F$ of order $2$ with $\on{Lip}(\F)\leq \lambda_b$. By Observation \ref{obs2}(ii), each $Y_k$ is open, so we see that the first derivative $(X\cup Y)'=X'$. Using the Mazurkiewicz-Sierpi\'nski theorem, we see that $X\cup Y$ is homeomorphic with $K$.\\
Case 2. $K^{(\infty)}$ is nonempty and not open in $K$.\\
Then there is a sequence $(x_n)\subset K$ of isolated points which converges to some $y_0\in K^{(\infty)}$. Indeed, by our assumption, there is a sequence $(x_n')\subset K\setminus K^{(\infty)}$ and $y_0\in K^{(\infty)}$ such that $x_n'\to y_0$. However, for any $k\in\N$, there is a clopen set $V\ni x'_k$ which is disjoint with $K^{(\infty)}$ and has diameter $\leq \frac{1}{k}$. In particular, $V$ is countable compact metrizable space, hence scattered, and it must have an isolated point $x_k$, which is also isolated in $K$. Then $x_k\to y_0$. Now let $X\cup Y$ be a $(\La_{\max},(b,p),s)$-space. Since $K':=K\setminus \{x_k:k\in\N\}$ is compact,  uncountable $0$-dimensional and metrizable, by Proposition \ref{1813}, we can find $M\subset X$ so that (\ref{1812}) is satisfied and a homeomorphism $ h:K'\to M$ so that $h(y_0)=x_\omega$. Enumerate elements of $\bigcup_{k\in\N}Y_k$ by $\{y_k:k\in\N\}$, and extend $h$~by adjusting $h(x_k):=y_k$ for $k\in\N$. It is easy to see that $h:K\to M\cup Y$ is homeomorphism. Also, by Theorem \ref{scattereddok1}(2), $M$ is the attractor of some Banach GIFS of order $2$ with $\on{Lip}(\F)\leq\lambda_b$.\\
Case3. $K^{(\infty)}$ is nonempty and open in $K$.\\
Then $K':=K\setminus K^{(\infty)}$ is countable compact space, hence scattered, and we can find a proper tree $\Lambda'$ as in Case 1. Now let 
$$
\Lambda:=\{1\ha\eta:\eta\in\La_{\max}\}\cup\{i\ha\eta:2\leq i\leq \omega,\;(i-1)\ha\eta\in\La'\}.
$$
Clearly, $\La$ is a proper tree and (according to the notation from Lemma \ref{1822}), $\La_1[1]=\La_{\max}$ and $\La_2[2]=\La'$. Now let $X\cup Y$ be a $(\La,(b,p),s)$-space. Then by Lemma \ref{1822}, $X_1$ is $(\La_{\max},b^1,s)$-space and $\tilde{X}_1$ is $(\La',b^1,s)$-space. In the same way as in Case 1, we can show that $\tilde{X}_1\cup Y$ is homeomorphic to $\tilde{X}_1$ (and to $K'$).
 Also, as $K^{(\infty)}$ is a Cantor space, it is homeomorphic to a subset $N\subset X_1$ which satisfies (\ref{1812}) (see Proposition \ref{1813}), and then $K=K'\cup K^{(\infty)}$ is homeomorphic to $(\tilde{X}_1\cup Y)\cup N$ (as the underlying sets are clopen). Finally, as $\tilde{X}_1$ and $N$ are attractors of some GIFSs $\F_1,\F_2$ of order $2$ with $\on{Lip}(\F_i)\leq\lambda_b$ (by Theorem \ref{scattereddok1}), $M:=\tilde{X}_1\cup N$ is {the} attractor of a Banach GIFS $\F$ of order $2$ with $\on{Lip}(\F)\leq\lambda$ (see Lemma \ref{filipnowy9}). 

\end{proof}

\begin{theorem}\label{m-scattereddok}
Let $Z=X\cup Y$ be a $(\La,(b,p),s)$-space and $M\subset X$ be the attractor of  some Banach GIFS~$\F$ of~order $2$ with $\on{Lip}(\F)\leq\lambda_b$ and such that $X_k\cap M\neq\emptyset$ for $2\leq k\leq \omega$.\\
(1) If $m\geq 2$, $p_1\geq 2$, $p_{n+1}=p_n^m$ for $n\in\N$ and $b$ is such that $\lambda_b< \frac{1}{2}$, then
\begin{itemize}
\item[(1a)] $M\cup Y$ is the attractor of some Banach GIFS $\G$ of order $m$ with $\on{Lip}(\G)\leq 2 \lambda_b$.
\item[(1b)] $M\cup Y$ is not the attractor of any weak GIFS of order $m-1$.
\end{itemize}
(2) If $p_1\geq 2$ and $p_{n+1}:=p_{n}^{n}$ for $n\in\N$, then $M\cup Y$ is not the attractor of any weak GIFS.
\end{theorem}
Theorem \ref{m-scattereddok} will follow from the next lemma:
\begin{lemma}\label{mdok1}
Let $Z=X\cup Y$ be a $(\La,(b,p),s)$-space, $M\subset X$ be such that $M\cap X_k\neq\emptyset$ for $2\leq k\leq\omega$ and $m\in\N$.\\
(1) If $\lim_{n\to\infty}\frac{p_n^m}{p_{n+1}}=0$, then $M\cup Y$ is not the attractor of any weak GIFS of order $m$.\\
(2) If  $M$ is the attractor of  some Banach GIFS $\F$ of order $m$ with $\on{Lip}(\F)< \frac{1}{2}$ and for any $n\in\N$, $p_{n+1}\leq p_n^m$, and $b$ is such that $\lambda_b<\frac{1}{2}$, then $M\cup Y$ is the attractor of some Banach GIFS $\G$ of order $m$ such that $\on{Lip}(\G)\leq 2 \max\{\on{Lip}(\F),\lambda_b\}$.
\end{lemma}
We show that Theorem \ref{m-scattereddok}(1) follows from the above Lemma. Observe that if $p_1\geq 2$ and $p_{n+1}=p_n^m$, then $\lim_{n\to\infty}\frac{p_n^{m-1}}{p_{n+1}}=\lim_{n\to\infty}\frac{1}{p_n}=0$. Hence (1b) follows from Lemma \ref{mdok1}(1). Clearly, (1a) follows from Lemma \ref{mdok1}(2). Finally, if $p_1\geq 2$ and $p_{n+1}=p_n^n$, then for any $m\in\N$, we have $\lim_{n\to\infty}\frac{p_n^m}{p_n^n}=\lim_{n\to\infty}\frac{1}{p_n^{n-m}}=0$, so (2) follows from Lemma \ref{mdok1}(1).\\
In the {remaining part} of this section we prove Lemma \ref{mdok1}.

For every $k\in\oN$, define $M_k=X_k\cap M$. If $M_1=\emptyset$, then define $\pi_1:M\cup Y\to M$ by
$$
\pi_1(x)=\left\{\begin{array}{ccl}x&\mbox{if}&x\in M,\\
z_{\max\{k,2\}}&\mbox{if}&x\in Y_k,\;k\in\N\end{array}\right.
$$
where $z_k$ is a fixed element of $M_k$.
We will show that $\on{Lip}(\pi_1)\leq 2$. Take $x,y\in M\cup Y$ and consider (most important) cases:\\
Case 1. $x\in M_k$ and $y\in Y_k$ for some $k\in\N$, $k\geq 2$. Then by Observation \ref{obs2}(iv),
$$d(\pi_1(x),\pi_1(x))\leq \on{diam}(X_k)\leq\on{dist}(X_k,Y_k)\leq \on{dist}(M_k,Y_k)\leq d(x,y).$$
Case 2. $x\in Z_k$ and $y\in Z_j$ for $k<j\leq\omega$. By Observation~\ref{obs2}(ii),(iii), we have
$$
d(\pi_1(y),\pi_1(x))\leq \on{diam}{\Bigg(}\bigcup_{k\leq j\leq\omega} Z_j{\Bigg)}\leq $$ $$\leq \frac{5}{4}b_{k-1}\leq 2\cdot\frac{2}{3}b_{k-1}\leq 2\on{dist}(Z_k,Z_j)\leq 2 d(x,y).
$$
Case 3. Case 1 and Case 2 do not hold. Then $d(\pi_1(x),\pi_1(y))\leq d(x,y)$ clearly holds.\\
Similarly we can define a projective map $\pi_2:M\cup Y\to Y$ with $\on{Lip}(\pi_2)\leq 2$.\\
If $M_1\neq\emptyset$ we can also define appropriate projections $\pi_1,\pi_2$.\\
In view of Lemmas \ref{nl1} and~\ref{filipnowy9}, to complete the proof of (2), it suffices to show that $Y$ is {the} attractor of some GIFS $\G$ of order $m$ with~$\on{Lip}(\G)\leq\lambda_b$.\\
For every $k\in \oN$, let 
$$G_k:=\bigcup\{Y_{i_1}\times{\dots}\times Y_{i_m}:\max\{i_1,{\dots},i_m\}=k\}.$$
Then $Y^m=\bigcup_{k\in\oN}G_k$ and for $k\in\N$, the cardinality $|G_k|\geq p_k^m$. Hence by our assumptions, for every $2\leq k\leq \omega$, there exists a surjection $h_k:G_{k-1}\to Y_{k}$ (we set $G_{\omega-1}:=G_\omega$). Finally if $x\in Y^m$, then we set 
$$F(x):=h_k(x) \;\;\mbox{if}\;\;x\in G_k.$$  
Clearly, $F:Y^m\to (Y\setminus Y_1)$ is surjection. Now we show that
\begin{equation}\label{fc4}
\on{Lip}(F)\leq \lambda_b.
\end{equation}
Take distinct $x=(x_1,{\dots},x_m),y=(y_1,{\dots},y_m)\in Y^m$ and consider cases:\\
Case 1. { For some $2 \leq k < \omega$ it holds $x,y\in G_{k-1}$}.\\
Then by Observation \ref{obs2}(ii),
$$
d_m(x,y)\geq \min\big\{\min\{\on{dist}(Y_i,Y_j):i<j\leq k-1\},\min\{2b_j:j\leq k-1\}\big\}\geq $$ $$\geq \min\left\{\frac{2}{3}b_{k-2},2b_{k-1}\right\}=2b_{k-1}
$$
and by (\ref{good})
$$
d(F(x),F(y))\leq 2b_k\leq\lambda_bp_k 2b_{k-1}\leq \lambda_bd_m(x,y).
$$
Case 2. { For some $2\leq k<j\leq\omega$, it holds} $x\in G_{k-1}$ and $y\in G_{j-1}$. Then there exists $i\in \{1,{\dots},m\}$ such that $y_i\in Y_{j-1}$. Then $x_i\in Y_n$ for some $n\leq k-1$, hence by (\ref{181a}), (\ref{good}) and Observation \ref{obs2}(ii),(iii),
$$
d(F(x),F(y))\leq\on{diam}(Y_k\cup Y_{j})\leq \frac{5}{4}b_{k-1}\leq\frac{5}{4}b_n\leq\lambda_b\frac{2}{3}b_{n-1}\leq $$ $$\leq \lambda_b\on{dist}(Y_n,Y_{j-1})\leq\lambda_bd(x_i,y_i)\leq\lambda_b d_m(x,y).
$$
Hence the proof of (\ref{fc4}) is finished.\\
Finally, let $G_1,{\dots},G_{p_1}:Y\to Y_1$ be constant maps so that $\bigcup_{i=1}^{p_1}G_i(Y)=Y_1$. Then $\on{Lip}(G_i)=0\leq\lambda_b$ and, in view of Lemma \ref{nl1}, $Y$ is the attractor of a GIFS of order $m$ with $\on{Lip}(\G)\leq\lambda_b$. The proof of Lemma~\ref{mdok1}(2) is finished.
\newline
Now we prove Lemma~\ref{mdok1}(1). It is enough to prove that for every generalized weak contraction $F:(M\cup Y)^m\to Z$,
\begin{equation}\label{nc5}
\lim_{n\to\infty}\frac{|F((M\cup Y)^m)\cap Y_n|}{p_n}=0.
\end{equation}
We start with the additional ''structural'' observation concerning $(\La,b,s)$-spaces.

\begin{observation}\label{obs3} Let $X$ be a $(\La,b,s)$-space.
For every $k\in\N$ and $i\in\N\cup\{0\}$, there exist $X_{k,i}^1,{\dots},X_{k,i}^{2^i}$ such that $X_k=\bigcup_{j=1}^{2^i}X_{k,i}^j$ and $\on{diam}(X_{k,i}^j)\leq b_{k+i}$.
\end{observation}

\begin{proof}By Lemma \ref{1822}{(d)} we see that $X_k$ is a $(\La',b^k,s)$-space of diameter $\leq b_k$, which can be divided into $(\La'_1[(k,1)],b^{k+1},s)$-space and $(\La'_2[(k,2)],b^{k+1},s)$-space of diameters $\leq b_{k+1}$. If both of these sets are not singletons, then we can divide each of them into next two, each of diameter $\leq b_{k+2}$. Proceeding in this way, for each $i$ we can divide $X_k$ into $2^i$ spaces, each of the diameter $\leq b_{k+i}$ (however, some of these sets can be empty, as during the procedure we can have singletons at some stages).
\end{proof}

Fix $n\in\N$. We will estimate $|F((M\cup Y)^m)\cap Y_n|$.\\
Let $A_1,{\dots},A_m$ be sets so that every $A_l$ is of one of the following forms (see Observation \ref{obs3} for the notation in the second option):
\begin{itemize}
\item[(a)] $A_l$ is a singleton;
\item[(b)] $A_l=X^{j}_{k,n-k}\cap M$ for some $1\leq k\leq n$ and $1\leq j\leq 2^{n-k}$;
\item[(c)] $A_l=\bigcup_{n+1\leq j\leq\omega}((X_j\cap M)\cup Y_j)$;
\item[(d)] $A_l=Y_n$.
\end{itemize}
We will observe that 
\begin{equation}\label{nc66}
\left|F\left(A_1\times{\dots}\times A_m\right)\cap Y_n\right|\leq 1
\end{equation}
Indeed, suppose that it is not true. Then we can find $x,y\in A_1\times{\dots}\times A_m$ such that $d(F(x),F(y))\geq 2b_n$. Now consider two cases:\\
Case 1. All $A_1,{\dots},A_m$ are of the form (a), (b) or (c). Since $F$ is a weak contraction and by earlier calculations (see Observation \ref{obs2}(vii)) we get a contradiction since:
$$2b_n\leq d(F(x),F(y)) < d_m(x,y) \leq \on{diam}(A_1\times{\dots}\times A_m)= $$ $$=\max\{\on{diam}(A_l):l=1,{\dots},m\}\leq \max\{b_n,2b_n\}= 2b_n.$$
Case 2. $A_{i_1}={\dots}=A_{i_k}=Y_n$ for some $i_1,{\dots},i_k$. Assume, without loss of generality, that $i_1=1,{\dots},i_k=k$. Then, setting $x=(x_1,{\dots},x_m),\;y=(y_1,{\dots},y_m)$, we can assume that 
$$
\max\{d(x_i,y_i):i=1,{\dots},k\}= j2b_n,\;\mbox{for some}\;\;2\leq j\leq p_n-1.
$$ 
Supposing otherwise we would obtain a contradiction with contractivity of $F$ similarly as in Case 1.\\
Moreover, we assume that $x,y$ are chosen so that $j$ is minimal in this sense, i.e., there are no 
$x',y'\in A_1\times{\dots}\times A_m$ with $F(x'),F(y')\in Y_n$, $F(x')\neq F(y')$ and $\max\{d(x'_i,y'_i):i=1,{\dots},k\} <j2b_n$. Now let $i=1,{\dots},k$. If $d(x_i,y_i)\leq 2b_n$, then set $z_i:=x_i$, and if $d(x_i,y_i)>2b_n$, then let $z_i$ be an element of $Y_n$ so that $d(x_i,z_i)<d(x_i,y_i)$ and $d(z_i,y_i)<d(x_i,y_i)$ (the existence of $z_i$ is guaranteed by Definition \ref{1823}(iia)). For $i=k+1,{\dots},m$ set $z_i:=x_i$. Then $z:=(z_1,{\dots},z_m)\in A_1\times{\dots}\times A_m$, $
\max\{d(x_i,z_i):i=1,{\dots},k\} <j2b_n
$ and $
\max\{d(y_i,z_i):i=1,{\dots},k\} <j2b_n
$. 
The following conditions can hold:\\
Case 2a. $F(z)\in Y_n$. Then by the minimality of $j$, $F(z)=F(x)$ and $F(z)=F(y)$ -- it is a contradiction with $F(x) \neq F(y)$.\\
Case 2b. $F(z)\in Z_i$ for some $i\neq n$, then by Observation \ref{obs2}(ii) and (\ref{good}), we arrive to a contradiction:
$$d(F(z),F(x))\geq \on{dist}(Z_n,Z_i)\geq \frac{1}{2}\min\{b_0,{\dots},b_{n-1}\}=$$ $$=\frac{1}{2}b_{n-1}\geq 2p_nb_n\geq 2jb_n> d_m(z,x).$$
Case 2c. $F(z)\in X_n$, then we get a contradiction since $$d(F(z),F(x))\geq\on{dist}(X_n,Y_n)= 2p_nb_n> d_m(z,x).$$

All in all, we proved (\ref{nc66}).

Notice that (see Observation \ref{obs3})
$$M\cup Y=$$ $$={\Bigg(}\bigcup_{1\leq i\leq n-1}Y_i{\Bigg)}\cup Y_n\cup {\Bigg(}\bigcup_{1\leq k\leq n}\bigcup_{1\leq j\leq 2^{n-k}}(X^{j}_{k,n-k}\cap M){\Bigg)}\cup{\Bigg(}\bigcup_{n+1\leq j\leq\omega}(X_j\cap M)\cup Y_j{\Bigg)}.
$$
Hence for $n\in\N$
$$
|F((M\cup Y)^m)\cap Y_n|\leq ((p_1+{\dots}+p_{n-1})+1+(1+2+{\dots}+2^{n-1})+1)^m\leq $$ $$\leq (1+p_1+{\dots}+p_{n-1}+n2^n)^m.
$$
In order to prove (\ref{nc5}), in view of the assumption $\lim_{n\to\infty}\frac{p_n^m}{p_{n+1}}=0$, it is enough to show that for some $M>0$ and all $n\in\N$, 
\begin{equation}\label{nc11}
\frac{1+p_1+{\dots}+p_n+(n+1)2^{n+1}}{p_n}<M.
\end{equation}
By our assumptions, there exists $n_0$ such that for $n\geq n_0$, $\frac{p_n}{p_{n+1}}\leq \frac{p^m_n}{p_{n+1}}\leq \frac{1}{4}$. Observe that 
\begin{equation}\label{nc10}
\frac{p_{n_0+1}+{\dots}+p_{n_0+k}}{p_{n_0+k}}\leq 2.
\end{equation}
Indeed, for $k=1$ it clearly holds, and if it holds for $k\geq 1$, then
$$
\frac{p_{n_0+1}+{\dots}+p_{n_0+k+1}}{p_{n_0+k+1}}=\frac{p_{n_0+1}+{\dots}+p_{n_0+k}}{p_{n_0+k}}\cdot\frac{p_{n_0+k}}{p_{n_0+k+1}}+1\leq 2\frac{1}{2}+1\leq 2.
$$
Similarly we can show that for some $N$ and every $n\in\N$, $\frac{(n+1)2^{n+1}}{p_n}\leq N$ (we can take $N=\frac{8}{p_1}$).
Hence for every $k\in\N$,
$$
\frac{1+p_1+{\dots}+p_{n_0+k}+(n_0+k+1)2^{n_0+k+1}}{p_{n_0+k}}\leq \frac{1+p_1+{\dots}+p_{n_0}}{p_{n_0}} + 2+N.
$$
We proved (\ref{nc11}), and proof of Lemma \ref{mdok1}(1) is finished.

\section{$(\La^\omega,b,Z)$-spaces as attractors of GIFSs}

The main result of this section is the following:

\begin{theorem}\label{X-scattered}
Let $\lambda\in(0,1)$ and $Z$ be a compact, connected metric space which is the attractor of some Banach GIFS $\F$ of order $2$ with $\on{Lip}(\F)\leq \frac{1}{3}\lambda$. 
Then there exists a compact metric space $X$ such that:
\begin{itemize}
\item[(a)] each connected component of $X$ is a {similarity} copy of $Z$,
\item[(b)] $X$ is the attractor of some Banach GIFS $\F$ of order $2$ with $\on{Lip}(\F)\leq \lambda$,
\item[(c)] $X$ is not a topological fractal;
\end{itemize}
Additionally, if $Z$ is a subset of a {normed} space $H$, then $X$ can be taken so that
\begin{itemize}
\item[(d)] $X\subset H$.
\end{itemize}
\end{theorem}
Theorem \ref{X-scattered} will follow from its qualitative version Theorem \ref{Xdok} {together with Theorem \ref{fvf3}}

Recall that in Subsection~\ref{s31}, for each countable limit ordinal $\delta_0$ and $\alpha\leq\delta_0$, we defined a certain sequence $(\alpha_n)$. Consider the case $\delta_0=\omega$. Then we can assume that sequences $(\alpha_n)$ are of the following forms:
\begin{itemize}
\item[*] if $\alpha = k \in\N$, then $(\alpha_n) = (0,1,{\dots},k-1,k-1,k-1,{\dots})$ i.e., $k_n=\min\{n-1,k-1\}$;
\item[*] if $\alpha = \omega$, then $(\omega):=(\alpha_n)=(0,1,2,{\dots})$, i.e., $\omega_n=n-1$.
\end{itemize}
It is easy to see that (a) and (b) from Section \ref{s31} are satisfied.\\
Observe that if $X$ is a $(\La^\omega,b,Z)$ space and $Z$ is connected, then each connected component of $X$ is a {similarity} copy of $Z$. Hence Theorem \ref{X-scattered} follows from the following:
\begin{theorem}\label{Xdok}
In the above framework, let $X$ be a $(\La^\omega,b,Z)$-space, where $Z$ is connected.\\
(1) If $Z$ is the attractor of a Banach GIFS $\F$ of order $2$, then $X$ is the attractor of a Banach GIFS $\G$ of order $2$ with $\on{Lip}(\G)\leq\max\{3\on{Lip}(\F),\lambda_b\}$.\\
(2) $X$ is not a topological fractal.
\end{theorem}

In the {remaining part} of this section we prove Theorem \ref{Xdok}. 
Directly by definition of families $\Lambda^\alpha$ for $\alpha \leq \omega$ and the above assumption, we can see that for every $k\geq 2$ (recall Proposition \ref{1}),
\begin{itemize}
\item[(A)] $\tL^{k-1}=\{\omega\}\cup{\big(}\bigcup_{i\leq k-1}i\ha \tL^{i-1}{\big)}\cup{\big(}\bigcup_{\omega>i\geq k}i\ha \tL^{k-2}{\big)}$;
\item[(B)] $\tL^\omega=\{\omega\}\cup{\big(}\bigcup_{i\leq k-1}i\ha\tL^{i-1}{\big)}\cup{\big(}\bigcup_{\omega>i\geq k}i\ha(k-1)\ha \tL^{k-2}{\big)}\cup$\\  $\cup{\big(}\bigcup_{\omega>i\geq k}(\bigcup_{j\neq k-1}i\ha j\ha \tL^{(i-1)_j}\cup\{(i,\omega)\}){\big)}$.
\end{itemize}
Now define $R_k:\tL^\omega\to k\ha \tL^{k-1}$ by
$$
R_k(\xi):=\left\{\begin{array}{ccc}k\ha i\ha \eta&\mbox{if}&\xi=i\ha\eta,\;i\leq k-1,\\
k\ha i\ha \eta&\mbox{if}&\xi=i\ha(k-1)\ha \eta,\;k\leq i<\omega,\\
k\ha i\ha \omega&\mbox{if}&\xi=i\ha \eta,\;k\leq i<\omega,\;\eta_1\neq k-1,\\
k\ha \omega &\mbox{if}&\xi=\omega.
\end{array}\right.
$$ 
By (A) and (B) we see that the map $R_k$ is well defined and $R_k(\tL^\omega)=k\ha\tL^{k-1}$. Now define the map $g_k:X\to X_k$ in the following way: if $\xi\in \tL^\omega$ is of the form $\xi=i\ha\eta$ for $i\leq k-1$ or $\xi=i\ha(k-1)\ha\eta$ for $\omega>i\geq k$ or $\xi=\omega$, then let $g_{k}\vert_{X_\xi}$ be a {similarity} transformation of $X_\xi$ onto $X_{R_k(\xi)}$ so that, if $\xi$ ends with $\omega$, then $g_k(x_{\xi})=x_{R_k({\xi})}$ (condition (Z3) guarantees that we can choose such {similitude}). If $\xi\in\tL^\omega$ is of the form $\xi=i\ha \eta$ for $i\geq k$ and $\eta_1\neq k-1$, then let
 $g_k\vert_{X_\xi}$ be a~constant map from $X_\xi$ to $X_{R_k(\xi)}$ such that $g_k(X_\xi) = \{x_{R(\xi)}\}$.\\
Clearly, $g_k(X)=X_k$. We will show that $\on{Lip}(g_k)\leq\frac{1}{4}\lambda_b$.\\
If $\xi\in \tL^\omega$ is of the form $\xi=i\ha\eta$ for $i\leq k-1$ or $\xi=i\ha(k-1)\ha\eta$ for $i\geq k$ or $\xi=\omega$, then\\
-- if $\xi$ does not end with $\omega$, then
$$
\on{diam}(X_\xi)=b_{l(\xi)}+b_{l(\xi)+1}\;\;\mbox{and}\;\;\on{diam}(g_k(X_\xi))=$$ $$=b_{l(R_k(\xi))}+b_{l(R_k(\xi))+1}\leq b_{l(\xi)+1}+b_{l(\xi)+2}.
$$
-- if $\xi$ ends with $\omega$, then letting $\xi$ be such that $\xi=\tilde{\xi}\ha\omega$, we get
$$
\on{diam}(X_\xi)=b_{l(\tilde{\xi})+1},\;\mbox{and}\;\on{diam}(g_k(X_\xi))=b_{l(\tilde{R_k(\xi)})+1}\leq b_{l(\tilde{\xi})+2}.
$$
If $\xi\in\tL^\omega$ is of different form than those above, then $\on{diam}(g_k(X_\xi))=0$.\\
Hence $\on{Lip}(g_k\vert X_\xi)\leq M_b\leq\frac{1}{4}\lambda_b$ for all $\xi\in\tL^\omega$.\\
Now let $x,y\in X$ be such that for some $\beta\in\La^\omega$ and $1\leq p<q\leq\omega$, $x\in X_{\beta\ha p}$ and $y\in X_{\beta\ha q}$, but if $q=\omega$, then we assume $y=x_{\beta\ha\omega}$, and such that $g_k(x)\neq g_k(y)$. Then by (\ref{181b}),
$$
\frac{1}{20}\lambda_bd(x,y)\geq \frac{1}{20}\lambda_b\on{dist}(X_{\beta\ha p},X_{\beta\ha q})\geq$$ $$\geq \frac{1}{20}\lambda_b(b_{l(\beta)+p-1}-2b_{l(\beta)+p}-b_{l(\beta)+p+1})\geq b_{l(\beta)+p}.
$$
Now consider cases:\\
Case 1. $\beta=i\ha\eta$ for $i\leq k-1$. Then
$$
d(g_k(x),g_k(y))\leq\on{diam}{\Bigg(}\overline{\bigcup_{p\leq j< \omega}X_{k\ha i\ha\eta\ha j}}{\Bigg)}\leq b_{k+l(\beta)+p-1}\leq b_{l(\beta)+p}\leq\frac{1}{20}\lambda_bd(x,y).
$$ 
Case 2. $\beta=i\ha(k-1)\ha\eta$ for some $k \leq i < \omega$. Then
$$
d(g_k(x),g_k(y))\leq \on{diam}{\Bigg(}\overline{\bigcup_{p\leq j < \omega}X_{k\ha i\ha\eta\ha j}}{\Bigg)}\leq $$ $$\leq
b_{k+i+l(\eta)+p-1}=b_{l(\beta)+p}\leq \frac{1}{20}\lambda_bd(x,y).
$$ 
Case 3. $\beta=i$ for $k\leq i<\omega$ and $p=k-1$ or $q=k-1$. Then
$$
d(g_k(x),g_k(y))\leq \on{diam}{\Bigg(}\overline{\bigcup_{1\leq j<\omega}X_{(k,i,j)}}{\Bigg)}\leq b_{k+i} \leq b_{l(\beta)+p}\leq \frac{1}{20}\lambda_bd(x,y).
$$
Case 4. $\beta=\emptyset$. Then 
$$d(g_k(x),g_k(y))\leq\on{diam}{\Bigg(}\overline{\bigcup_{p\leq j< \omega}X_{(k,j)}}{\Bigg)}\leq  b_{k+p-1} \leq b_{l(\beta)+p}\leq \frac{1}{20}\lambda_bd(x,y).
$$
Finally, assume that $y\in X_{\beta\ha\omega}$. By Definition \ref{183}(v), Lemma \ref{1822}{(b)} and what we already proved, we have
$$
d(g_k(x),g_k(y))\leq d(g_k(x),g_k(x_{\beta\ha\omega}))+d(g_k(x_{\beta\ha\omega}),g_k(y))\leq $$ $$\leq  \frac{1}{20}\lambda_b d(x,x_{\beta\ha\omega})+\frac{1}{20}\lambda_b d(x_{\beta\ha\omega},y)
\leq \frac{2}{20}\lambda_bd(x,y)+\frac{3}{20}\lambda_bd(x,y) = \frac{1}{4}\lambda_bd(x,y).
$$
Now by Lemma \ref{1814} and Remark \ref{1814a}, there is a map $F:X\times X\to X$ such that $\on{Lip}(F)\leq \frac{1}{2}\lambda_b$ and $F(X\times X)=\{x_\omega\} \cup \bigcup_{1\leq k<\omega}g_k(X)= \{x_\omega\} \cup \bigcup_{2\leq k<\omega}X_k$.\\
Since $Z$ is the attractor of a Banach GIFS $\F$ of order $2$ and $X_1$ and $X_\omega$ are {similarity} copy of~$Z$, there are GIFSs $\F_1=\{f_1,{\dots},f_n\}$ and $\F_\omega=\{g_1,{\dots},g_n\}$ on $X_1$ and $X_\omega$, respectively, of order $2$, with $\on{Lip}(\F_1),\on{Lip}(\F_\omega)=\on{Lip}(\F)$. Observe that the maps $\pi_1:X\to X_1$ and $\pi_\omega:X\to X_\omega$ given by
$$
\pi_1(x):=\left\{\begin{array}{ccc}x&\mbox{if}&x\in X_1\\
x_{(1, \omega)}&\mbox{if}&x\notin X_1\end{array}\right. ;
\;\;\;\;\;\pi_\omega(x):=\left\{\begin{array}{ccc}x&\mbox{if}&x\in X_\omega\\x_\omega&\mbox{if}&x\notin X_\omega\end{array}\right. .
$$
It is routine to check that $\on{Lip}(\pi_1),\on{Lip}(\pi_\omega)\leq 3$ (we use Lemma \ref{1822}{(b)}). Finally, define $f'_k(x,y):=f_k(\pi_1(x),\pi_1(y))$ and $g'_k(x,y):=g_k(\pi_\omega(x),\pi_\omega(y))$ and observe that $\mH=\{f'_1,{\dots},f'_k,g'_1,{\dots},g'_k,F\}$ is a Banach GIFS with $\on{Lip}(\mH)\leq \max\{3\on{Lip}(\F),\frac{1}{2}\lambda_b\}$ and $X$ is its { attractor}. This ends the proof of (1).\\
Now we prove (2). It is easy to see that $X/R_c$ is a scattered space with the height $\omega$. By the result of~Nowak and~Theorem~\ref{quotient}, $X$ is not a topological fractal.

\begin{example}{
Theorem \ref{X-scattered} gives us a way of constructing many mutually nonhomeomorphic GIFSs fractals which are not topological fractals. Indeed, if $Z_1,Z_2$ are not homeomorphic, then~spaces which have all components homeomorphic with $Z_1$ and $Z_2$, respectively, are not homeomorphic.\\
For example, for $Z$ we can take any cube $I_n=[0,1]^n$. Thanks to Lemma~\ref{nl2}, we obtain in this way fractals of GIFSs defined on the whole Euclidean spaces $\R^n$.\\
{ A slight modification of the Hilbert cube leads to an example in  $\ell^\infty$ space. Indeed, let 
$$
Z:=[0,1]\times\left[0,{\tfrac{1}{4}}\right]\times\left[0,{\tfrac{1}{4^2}}\right]\times{\dots}
$$
Then it is easy to see that $Z$ is the attractor of the GIFS $\F=\{f_i:i=0,1,2,3\}$, where
$$
f_i((x_k),(y_k)):=\left({\tfrac{1}{4}}x_1+{\tfrac{i}{4}},{\tfrac{1}{4}}y_1,{\tfrac{1}{4}}y_2,{\dots}\right),
$$
and $\on{Lip}(\F)=\frac{1}{4}$. Hence by Theorem \ref{X-scattered} and Remark \ref{rem:ell}, there is $X\subset \ell_\infty$ whose all connected components are {similarity} copy of $Z$, which is not a topological fractal and which is {the attractor of} some GIFS $\G$ on $\ell^\infty$ with $\on{Lip}(\G)\leq \frac{3}{4}$.\\
Similarly, starting with
$$
Z:=[0,1]\times\left[0,{\tfrac{1}{8}}\right]\times\left[0,{\tfrac{1}{8^2}}\right]\times{\dots}
$$
and using Theorem \ref{X-scattered} and Lemma \ref{nl2}, we obtain appropriate example in the Hilbert space $\ell^2$.}
}\end{example}

D'Aniello in \cite{D} showed that for any $n\in\N$ and $0< s\leq n$, there is a set $A\subset \R^n$ whose Hausdorff dimension $\on{dim}_H(A)=s$ and which is not the fractal generated by any weak IFS on $\R^n$. However, the sets she constructed were certain Cantor sets, hence topological fractals.\\ 
Theorem \ref{X-scattered} implies the following:
\begin{corollary}
If $1\leq s\leq n$, then there is a set $A\subset \R^n$ such that $\on{dim}_H(A)=s$ and which is not a topological fractal but is {the attractor} of some Banach GIFS on $\R^n$ of order $2$.
\end{corollary}
\begin{proof}It is enough to~take as~$Z$ a connected IFS {fractal} of Hausdorff dimension equal to $s$.\end{proof}

\begin{remark}{
It is easy to see that justification of Theorem \ref{Xdok}(2) is more general -- if $X$ is~a~$(\La,b,Z)$-space and $Z$ is connected, then the quotient space $X/R_c$ is homeomorphic with $(\La,b,s)$-space. Hence for any connected space $Z$ and any countable $\alpha$ with limit height, the space $(\La,b,Z)$-space is not a topological fractal.}
\end{remark}

\section{Topological properties of class of GIFSs' { attractors}}
It is well known (see for example \cite{BM} and {\cite{DS2}}) that (in reasonable metric spaces $X$) the class of attractors of weak IFSs is meager in $\mK(X)$. In this section we extend this result in some sense - we show that the class of all attractors of \textbf{Banach} GIFSs on a Hilbert space $H$ is meager.\\
Given a Hilbert space $H$ and $m\in\N$, let
$$
\mathcal{A}^{m}(H):=\{K\in\mK(H):K\;\mbox{is the attractor of a Banach GIFS of order $m$}\},
$$
$$
\mathcal{A}_w^{m}(H):=\{K\in\mK(H):K\;\mbox{is the attractor of a weak GIFS of order $m$}\},
$$
where $\mK(H)$ is the family of all nonempty and compact subsets of $H$. We consider it as a metric space with the Hausdorff metric ${h}$.\\
Finally, define
$$
\mA(H):=\bigcup_{m\in\N}\mA^{m}(H).
$$
That is, $\mA^m(H)$ and $\mA^m_w(H)$ are classes of all attractors of Banach and weak GIFSs of order $m$, respectively, and $\mA(H)$ is the class of all { attractors of Banach GIFSs} .
\begin{theorem}In the above frame:\\
(1) The set $\mA(H)$ is meager $F_\sigma$ in $\mK(H)$ and, in particular, typical compact subset of $H$ is not {the attrator of any} Banach GIFS.\\
(2) For every $m\in\N$, the set $\mA^{m+1}(H)\setminus \mA_w^m(H)$ is dense in $\mK(H)$.
\end{theorem}
\begin{proof}
We first prove (2). Let $K\in \mK(H)$ and fix $\varepsilon>0$. Then choose a finite set $K'=\{x_1,{\dots},x_k\}$ so that ${h}(K,K')<\frac{\varepsilon}{2}$, where ${h}$ is the Hausdorff metric on $\mK(X)$. Identifying a fixed line in $H$ with~$\R$, by~Theorem~\ref{m-scattered}, we can find a set $Y\in \mA^{m+1}(H)\setminus \mA^m_w(H)$. Then we replace each point $x_i$ in $K'$ by~appropriately small copy $Y_i$ of $Y$ so that $\on{diam}(Y_i)<\frac{\varepsilon}{2}$ and $\on{dist}(Y_i,Y_j)\geq \on{diam}(Y_l)$ for $i\neq j$ and $l$. Then, denoting by $K'':=\bigcup_{i=1}^kY_i$, we have ${h}(K,K'')\leq {h}(K,K')+{h}(K',K'')<\varepsilon$, and by Lemmas~\ref{filipnowy9},~\ref{nl3} and~Remark \ref{nr1}, $K''\in\mA^{m+1}(H)\setminus \mA^m_w(H)$. This ends the proof of (2).\\
Now we prove (1). For every $n,m\in\N$ and $\alpha<1$, denote by
$
\mathcal{A}_{n,\alpha}^{m}(H)$ the family of all attractors of Banach GIFSs $\F$ of order $m$ consisting of $\leq n$ maps and such that $\on{Lip}(\F)\leq\alpha$.
Clearly, $\mA(H)=\bigcup_{n,m,i\in\N}\mA^m_{n,\frac{i-1}{i}}(H)$. By (2), sets $\mA^m_{n,\alpha}(H)$ have empty interior. Thus it is enough to prove that each $\mA^m_{n,\alpha}(X)$ is closed. We will mimic the proof of \cite[Proposition 3.6]{DS2}. Choose a sequence $(K_k)\subset \mA^m_{n,\alpha}(X)$ which converges to some compact set $K$. Calculating, if needed, the same functions more than once, we may assume that for every $k\in\N$, there are maps $f^k_1,{\dots},f^k_n:(K_k)^m\to K_k$ so that $\on{Lip}(f^k_i)\leq\alpha$ and $K_k=\bigcup_{i=1}^nf^n_i((K_k)^m)$.
\newline
Now let $C$ be the closed convex hull of the compact set $K\cup \bigcup_{k\in\N}K_k$. Then $C$ is compact and convex subset of a Hilbert space, so there exists a retraction $r:X\to C$ with $\on{Lip}(r)=1$ ($r(x)$ can be the nearest point to $x$ in $C$). With the use of the Kirszbraun--Valentine theorem, extend each $f_i^k$ to a map $\overline{f}^k_i:H^m\to H$, with the condition $\on{Lip}(\overline{f}^k_i)\leq \sqrt{m}$ (see the proof of Lemma \ref{nl2}). Finally, set $\tilde{f}^k_i:=(r\circ \overline{f}^k_i)\vert_{C^m}$. In particular, $\tilde{f}^k_i:C^m\to  C$. Now since $C$ is compact and $\on{Lip}(\tilde{f}^k_i)\leq \sqrt{m}$, for every~$i=1,{\dots},n$, the sequence $(\tilde{f}^k_i)_{k\in\N}$ satisfies the assumptions of the Arzel\'a--Ascoli theorem. Hence there is a subsequence $(k_j)$ (which can be appropriate for all $i$) and a map ${f}_i:C^m\to H$ such that $(\tilde{f}^{k_j}_i)$ converges uniformly to ${f}_i$. For simplicity of notation, we assume that $k_j=j$.\\
To end the proof, it is enough to show that $\on{Lip}(f_i\vert_K)\leq \alpha$ and $K=\bigcup_{i=1}^nf_i(K^m)$. To see the first assertion, take $x,y\in K^m$. As $(K_{j})^m\to K^m$, we can find sequences $(x_j),(y_j)$ such that $x_j,y_j\in (K_{j})^m$, $x_j\to x$ and~$y_j\to y$. As $f^{j}_i(x_j)=\tilde{f}^j_i(x_j)\to f_i(x)$ and $f_i^{j}(y_j)\to f_i(y)$, we have
$$
d(f_i(x),f_i(y))=\lim_{j\to\infty}d(f^{j}_i(x_j),f^{j}_i(y_j))\leq\lim_{j\to\infty}\alpha d_m(x_j,y_j)=\alpha d_m(x,y).
$$
Hence $\on{Lip}(f_i)\leq \alpha$. In a similar way we can show that $f^{j}_i((K_j)^m)\to f_i(K^m)$, so $$K_{j}=\bigcup_{i=1}^nf^{j}_i((K_j)^m)\to \bigcup_{i=1}^nf_i(K^m).$$
On the other hand, $K_{j}\to K$, and hence $\bigcup_{i=1}^nf_i(K^m)=K$. This ends the proof.
\end{proof}

In view of mentioned results from \cite{BM} and \cite{DS2}, it is natural to ask is the set of all attractors of weak GIFSs is meager in $\mK(X)$. We leave it as an open problem:
\begin{problem}
Let $\mA_w(X)$ be the family of all sets $K\in\mK(X)$ which are attractors of some (possibly weak) GIFSs. Is $\mA_w(X)$ meager in $\mK(X)$?
\end{problem}

\textbf{Acknowledgements}\\
We would like to thank an anonymous referee for pointing out the way to extend point (2) of Theorem \ref{fvf3} (and, in turn, Theorem \ref{X-scattered}) from Hilbert spaces to normed spaces.

\end{document}